\documentclass[11pt,english]{article}

\usepackage[margin = 2.5cm]{geometry}

\usepackage{amsthm}
\usepackage{amsmath}
\usepackage{amssymb}
\usepackage{mathtools}
\usepackage{graphicx}
\usepackage[pdftex,colorlinks,backref=page,citecolor=blue,bookmarks=false]{hyperref}
\usepackage{color}
\usepackage{thm-restate}
\usepackage{enumitem}

\usepackage{caption}% http://ctan.org/pkg/caption
\usepackage[square,sort,comma,numbers]{natbib} 
\usepackage{setspace}
\usepackage{cleveref}
\theoremstyle{plain}

\newtheorem*{theorem*}{Theorem}
\newtheorem{theorem}{Theorem}[section]
\crefname{theorem}{Theorem}{Theorems}
\Crefname{theorem}{Theorem}{Theorems}

\newtheorem*{lemma*}{Lemma}
\newtheorem{lemma}[theorem]{Lemma}
\crefname{lemma}{Lemma}{Lemmas}
\Crefname{lemma}{Lemma}{Lemmas}

\newtheorem*{claim*}{Claim}
\newtheorem{claim}[theorem]{Claim}
\crefname{claim}{Claim}{Claims}
\Crefname{claim}{Claim}{Claims}

\newtheorem*{innerclaim*}{Claim}

\crefname{claim}{Claim}{Claims}
\Crefname{claim}{Claim}{Claims}

\newtheorem{proposition}[theorem]{Proposition}
\crefname{proposition}{Proposition}{Propositions}
\Crefname{proposition}{Proposition}{Propositions}

\newtheorem{corollary}[theorem]{Corollary}
\crefname{corollary}{Corollary}{Corollaries}
\Crefname{corollary}{Corollary}{Corollaries}

\newtheorem{conjecture}[theorem]{Conjecture}
\crefname{conjecture}{Conjecture}{Conjectures}
\Crefname{conjecture}{Conjecture}{Conjectures}

\crefname{question}{Question}{Questions}
\Crefname{question}{Question}{Questions}

\crefname{observation}{Observation}{Observations}
\Crefname{observation}{Observation}{Observations}

\crefname{example}{Example}{Examples}
\Crefname{example}{Example}{Examples}

\theoremstyle{definition}

\crefname{problem}{Problem}{Problems}
\Crefname{problem}{Problem}{Problems}

\crefname{definition}{Definition}{Definitions}
\Crefname{definition}{Definition}{Definitions}

\theoremstyle{remark}

\crefname{remark}{Remark}{Remarks}
\Crefname{remark}{Remark}{Remarks}

\usepackage{xpatch}
\xpatchcmd{\proof}{\itshape}{\normalfont\proofnamefont}{}{}
\newcommand{\proofnamefont}{}
\renewcommand{\proofnamefont}{\bfseries}

\usepackage{framed}

\newcommand{\remove}[1]{}

\newcommand{\ceil}[1]{
	\left\lceil #1 \right\rceil
}
\newcommand{\floor}[1]{
	\left\lfloor #1 \right\rfloor
}

\newcommand{\logstar}{\log^\star}

\newcommand{\HH}{\mathcal{H}}
\newcommand{\cH}{\mathcal{H}}
\newcommand{\cM}{\mathcal{M}}

\renewcommand{\P}{\mathcal{P}}
\newcommand{\cP}{\mathcal{P}}
\newcommand{\F}{\mathcal{F}}
\newcommand{\Q}{\mathcal{Q}}
\newcommand{\cQ}{\mathcal{Q}}

\newcommand{\cF}{\mathcal{F}}
\newcommand{\cA}{\mathcal{A}}

\newcommand{\eps}{\varepsilon}

\renewcommand{\setminus}{-}
\newcommand{\Ex}{\mathbb{E}}
\newcommand{\bR}{\mathbb{R}}
\renewcommand{\Pr}{\mathbb{P}}

\def \bd {\bar{d}}

\DeclareMathOperator{\polylog}{polylog}
\DeclareMathOperator{\sep}{sep}
\newcommand{\wsep}{\sep^*}

\begin{document}

\title{Separating path systems of almost linear size}
\author{Shoham Letzter\thanks{
		Department of Mathematics, 
		University College London, 
		Gower Street, London WC1E~6BT, UK. 
		Email: \texttt{s.letzter}@\texttt{ucl.ac.uk}. 
		Research supported by the Royal Society.
	}
}

\date{}
\maketitle

\begin{abstract}

	\setlength{\parskip}{\medskipamount}
    \setlength{\parindent}{0pt}
    \noindent

	A \emph{separating path system} for a graph $G$ is a collection $\mathcal{P}$ of paths in $G$ such that for every two edges $e$ and $f$, there is a path in $\mathcal{P}$ that contains $e$ but not $f$. We show that every $n$-vertex graph has a separating path system of size $O(n \log^* n)$. This improves upon the previous best upper bound of $O(n \log n)$, and makes progress towards a conjecture of Falgas-Ravry--Kittipassorn--Kor\'andi--Letzter--Narayanan and Balogh--Csaba--Martin--Pluh\'ar, according to which an $O(n)$ bound should hold.

\end{abstract}

\section{Introduction} \label{sec:intro}

	Given a set $A$ and a family $\F$ of sets, we say that $\F$ is a \emph{strongly-separating system} for $A$, or that $\cF$ \emph{strongly-separates} $A$, if for every two elements $a, b \in A$ there is a set $S \in \F$ satisfying $a \in S$ and $b \notin S$. Similarly, we say that $\F$ is a \emph{weakly-separating} system for $A$, or that $\F$ \emph{weakly-separates} $A$, if for every two elements $a, b \in A$ there is a set $S \in \F$ that contains exactly one of $a$ and $b$.\footnote{The notion of weak separation is sometimes referred to, simply, as \emph{separation}. We keep the word `weak' in the definition, as our focus will be on the strong version.} In particular, if $\cF$ strongly-separates $A$ then every element of $A$ is in some set in $\cF$, and if it weakly-separates $A$ then there is at most one element in $A$ not contained in any set in $\cF$.

	The notion of separation was introduced by R\'enyi \cite{renyi1961random} in 1961. It is easy to see that the smallest weakly-separating system for a set of size $n$ has size $\ceil{\log_2 n}$, and that the smallest strongly-separating system has size $(1 + o(1))\log_2 n$ (see \cite{spencer1970minimal}). The problem of determining the size of a smallest separating system becomes more interesting when one imposes restrictions on the members of the separating system.
	For example, separating systems with restrictions on the size of their members were studied in \cite{katona1966separating,wegener1979separating,ramsay1996minimal,kundgen2001minimal}, and different extremal questions about separating systems were studied in \cite{bollobas2007separating,bollobas2007separating-b,hansel1964nombre}.

	Another interesting direction is to consider separating systems where the ground set is a graph and the members of the separating system satisfy some graph theoretic properties; see \cite{cai1984separating,foucaud2013identifying,foucaud2013identifying-b,honkala2003cycles} for extremal examples \cite{tapolcai2012link,tapolcai2009monitoring,harvey2007non} for algorithmic examples. Our focus will be on \emph{separating path systems}, where the ground set is the set of edges of a given graph $G$ and the members of the separating system are paths in $G$. 
	Separating path systems arise naturally in the study of network design (see \cite{ahuja2009single,harvey2007non,tapolcai2012link}).
	In the context of extremal graph theory, this variant was proposed by Gyula Katona in the $5^{\text{th}}$ Eml\'ekt\'abla Workshop (2013). 
	Define $\wsep(G)$ to be the size of a smallest weakly-separating path system for $G$, and let $\wsep(n)$ be the maximum of $\wsep(G)$ over all $n$-vertex graphs $G$; define $\sep(G)$ and $\sep(n)$ analogously for strongly-separating path systems.
	Katona asked to determine $\wsep(n)$.

	Separating path systems were first studied by Falgas-Ravry--Kittipassorn--Kor\'andi--Narayanan and the author \cite{falgas2013separating}, who focused on the weak version, and, independently, by Balogh--Csaba--Martin--Pluh\'ar \cite{balogh2016path}, who focused on the strong version. Both groups considered path separation in various classes of graphs, including random graphs and trees. Moreover, both groups conjectured that the answer to Katona's question is $O(n)$.\footnote{Balogh--Csaba--Martin--Pluh\'ar's version is slightly stronger because they use the stronger notion of separation; we suspect that the two variants are equally hard.} 

	\begin{conjecture}[Falgas-Ravry--Kittipassorn--Kor\'andi--Letzter--Narayanan \cite{falgas2013separating} and Balogh--Csaba--Martin--Pluh\'ar \cite{balogh2016path}]
		There exists a constant $c$ such that every graph on $n$ vertices has a (weakly- or strongly-)separating path system of size at most $cn$.
	\end{conjecture}

	In fact, Falgas-Ravry--Kittipassorn--Kor\'andi--Letzter--Narayanan \cite{falgas2013separating} make the bold suggestion that perhaps, for weak separation, $c$ could be taken to be arbitrarily close to $1$, for sufficiently large $n$. If true, that would be tight, as can be seen by considering the complete graph $K_n$. In fact, determining $\wsep(K_n)$ and $\sep(K_n)$ seems hard; see Wickes \cite{wickes2022separating} for the best-known upper bound on $\wsep(K_n)$.

	The study of separating path systems of graphs is related to the well-researched topic of path decompositions. A famous conjecture of Gallai asserts that every connected graph on $n$ vertices can be decomposed into at most $\floor{\frac{n+1}{2}}$ paths; here by \emph{decomposition} we mean that each edge is covered exactly once. A fundamental result of Lov\'asz \cite{lovasz1968covering} proves the following slightly weaker version of Gallai's conjecture. 

	\begin{theorem}[\cite{lovasz1968covering}] \label{thm:lovasz}
		Every $n$-vertex graph can be decomposed into at most $n/2$ paths and cycles. Consequently, every $n$-vertex graph can be decomposed into at most $n$ paths.
	\end{theorem}

	\Cref{thm:lovasz} will play an important role in our proof. We sometimes use instead its following easy corollary.

	\begin{corollary} \label{cor:lovasz}
		Every $n$-vertex graph can be decomposed into at most $e(G)/d + n$ paths of length at most $d$.
	\end{corollary}

	\begin{proof}
		By \Cref{thm:lovasz}, there is a collection $\P$ of at most $n$ paths that decomposes the edges of $G$. Decompose each path in $\P$ into paths of length $d$ and at most one shorter path. We thus obtain a decomposition $G$ into paths of length at most $d$, at most $n$ of which have length shorter than $d$. The corollary follows as there can be at most $e(G)/d$ pairwise edge-disjoint paths of length $d$.
	\end{proof}

	The currently best-known bounds towards Gallai's conjecture are due to, independently, Dean--Kouider \cite{dean2000gallai} and Yan \cite{yan1998path}, who proved that every $n$-vertex graph can be decomposed into at most $\frac{2n}{3}$ paths. The conjecture was proved to hold for various families of graphs; see Blanch\'e--Bonamy--Bonichon \cite{blanche2021gallai} for a proof for connected planar graphs and the citations therein for other families of graphs.

	\Cref{thm:lovasz} implies an upper bound of $\sep(n) = O(n \log n)$. Indeed, given an $n$-vertex graph $G$, let $\cH$ be a collection of $O(\log n)$ subgraphs of $G$ that separate the edges of $G$; this exists by an aforementioned result regarding separating a set without underlying structure.
	Now decompose each $H \in \cH$ into at most $n$ paths. This yields a collection of $O(n \log n)$ paths separating $G$. Until now, this was the best-known general upper bound on $\sep(n)$. In this paper we improve it significantly, showing that $\sep(n) = O(n \logstar n)$. Here $\logstar n$ is the iterated logarithm, namely the minimum number of times the base-$2$ logarithm has to be applied, iteratively, to yield a result which is smaller than $1$.
	
	\begin{theorem} \label{thm:main}
		Every $n$-vertex graph has a strongly-separating path system of size $O(n \logstar n)$.
		That is, $\sep(n) = O(n \logstar n)$.
	\end{theorem}

	A key component in this paper is the use of sublinear expanders. These were introduced by Koml\'os--Szemer\'edi \cite{komlos1994topological,komlos1996topological} and recently various variants of them were used to prove some interesting results \cite{fernandez2022disjoint,fernandez2022nested,fernandez2022build,letzter2023immersion,haslegrave2021crux,haslegrave2021ramsey,haslegrave2021extremal,kim2017proof,liu2017proof,liu2020solution,liu2022clique,bucic2022erdos}. Specifically, we will use many ideas from a very recent paper by Buci\'c--Montgomery \cite{bucic2022erdos}, who proved that every $n$-vertex graph can be decomposed into $O(n \logstar n)$ cycles and edges. This is significant progress on the Erd\H{o}s--Gallai conjecture \cite{erdos1966representation}, asserting that every $n$-vertex graph can be decomposed into $O(n)$ cycles and edges. We will also use a tool from \cite{liu2020solution}, where Liu and Montgomery solved the so-called odd cycle problem, due to Erd\H{o}s--Hajnal \cite{erdos1966chromatic}.

	In the following section, \Cref{sec:overview}, we give an overview of the proof and an outline of the structure of the paper.
	We remark that, as our focus throughout the paper is on strong separation, from now on we will refer to it simply as \emph{separation}.

\section{Proof overview} \label{sec:overview}

	The main purpose of this section is to give an overview of the proof of \Cref{thm:main}.
	We will first describe a general strategy for finding small separating path systems (see \Cref{subsec:strategy}), we then show how sublinear expanders come into play (\Cref{subsec:expanders}), and we give a more detailed proof overview in \Cref{subsec:sketch}. At the end of the section we outline the structure of the paper (see \Cref{subsec:structure}) and define some notation (in \Cref{subsec:notation}).
		
	\subsection{A general strategy} \label{subsec:strategy}
		We start by describing a general strategy for finding a separating path system for a graph $G$ on $n$ vertices, which was used in \cite{falgas2013separating}. By \Cref{thm:lovasz}, there is a collection $\cP$ of at most $n$ paths that decomposes $G$; for an edge $e$, denote by $P(e)$ the unique path in $\cP$ that contains $e$. We claim that there is a collection $\cM$ of at most $3 n$ matchings that decomposes $G$, such that $\cP \cup \cM$ separates the edges of $G$. 

		Indeed, take $\cM$ to be a collection of $3n$ pairwise edge-disjoint matchings, whose edges intersect each path in $\cP$ in at most one edge, and that covers a maximum number of edges under these conditions. Suppose that $e$ is an edge in $G$ which is not covered by any of the matchings in $\cM$. Notice that there are at most $n-1$ matchings in $\cM$ that contain edges of $P(e)$, and at most $2(n-1)$ matchings in $\cM$ that contain an edge that intersects $e$. Hence, there is a matching $M \in \cM$ that contains neither edges of $P(e)$ nor edges that intersect $e$. Replacing $M$ by $M \cup \{e\}$, we reach a contradiction to the maximality of $\cM$, thus showing that $\cM$ covers all edges in $G$.
		Observe that $\cP \cup \cM$ separates the edges of $G$. Indeed, defining $M(e)$ to be the unique matching in $\cM$ that contains the edge $e$, we see that for any two edges $e$ and $f$, one of $P(e)$ and $M(e)$ contains $e$ but not $f$.
		
		Thus, if we could find a path $P_M$, for each $M \in \cM$, that contains $M$ but avoids all other edges of $\bigcup_{e \in M} P(e)$, then we would obtain a separating path system for $G$ whose size is $|\cP| + |\cM| \le 4n$.

		Of course, this strategy does not always work; otherwise we would have been able to improve on our main result, \Cref{thm:main}. Nevertheless, if $G$ has expansion properties and we impose additional restrictions on the matchings in $\cM$, then this strategy often succeeds.
		Indeed, in \cite{falgas2013separating} it was used to show that, with high probability, the random graph $G(n, p)$ has a weakly-separating path system of size $O(n)$ for every $p = p(n)$ (similar ideas appear in \cite{balogh2016path}, where the authors found much smaller strongly-separating path systems when $p \ge \frac{1000}{\log n}$). In this paper we apply it to a more general class of expanders. 

	\subsection{Expanders} \label{subsec:expanders}

		The exact notions of expansion that we shall use will be introduced later (see \Cref{sec:decompose-expanders,sec:random}). We mention that they are variants of the `sublinear expanders', which were introduced by Koml\'os and Szemr\'edi \cite{komlos1994topological,komlos1996topological}, and have since been modified in various ways to prove many interesting results \cite{fernandez2022disjoint,fernandez2022nested,fernandez2022build,letzter2023immersion,haslegrave2021crux,haslegrave2021ramsey,haslegrave2021extremal,kim2017proof,liu2017proof,liu2020solution,liu2022clique,bucic2022erdos}. As their name suggests, such expanders have a fairly weak expansion property. Their strength is in their omnipresence: every graph with a sufficiently large (but possibly still constant) average degree contains a sublinear expander. In fact, Buci\'c and Montgomery \cite{bucic2022erdos} recently showed that every graph can be decomposed (or almost decomposed, if we want the expanders to be `robust') into sublinear expanders that, on average, cover each vertex a constant number of times. This is very useful for us; it shows that, essentially, it suffices to be able to find separating path systems of linear size in expanders.

		With this decomposition-into-expanders result at hand, a rough plan would be to apply it to decompose a given graph into expanders, and then find a linear separating path system for each expander using the strategy above. If this plan could be realised as stated, then we would end up with a linear separating path system for the original graph. However, depending on the density of the expanders we find, we sometimes require them to be `robust' (or `somewhat robust'), and then some edges might remain uncovered. Moreover, in some situations, our methods separate almost all, but not all, edges of a given expander.
		Instead, we perform $O(\logstar n)$ steps, each time significantly decreasing the number of unseparated edges, until remaining with $O(n)$ edges that we can separate trivially by individual edges.

	\subsection{Proof sketch} \label{subsec:sketch}
		
		We now give a more concrete outline of our proof.
		Our main task is to prove that, given a graph $G$ on $n$ vertices and with average degree $d$, there is a subgraph $G_1$ on $O(n \polylog d)$ edges and a collection $\cP$ of $O(n)$ paths in $G$ that separates the edges of $G \setminus G_1$. 

		It now essentially suffices to separate the edges of $G_1$. Indeed, suppose that $\cQ$ is a collection of paths in $G$ that separates $G_1$.
		By \Cref{thm:lovasz}, there are collections $\cP'$ and $\cQ'$ of at most $n$ paths that decompose $G \setminus G_1$ and $G_1$. 
		We claim that $\cP \cup \cP' \cup \cQ \cup \cQ'$ separates the edges of $G$. To see this, let $e$ and $f$ be distinct edges in $G$. If both are in $G \setminus G_1$, then $\cP$ has a path containing $e$ but not $f$; if both are in $G_1$ then $\cQ$ has such a path; if $e$ is in $G \setminus G_1$ and $f$ is not, then $\cP'$ has a path through $e$ but not $f$; and finally if $e$ is in $G_1$ but $f$ is not then $\cQ'$ has such a path.

		Iterating the argument in the paragraph before last thus yields a separating path system for $G$ of size $O(n \logstar n)$.
		We complete this task in three steps.

		\paragraph{Step 1: dense expanders.}
			We show (in \Cref{lem:reduce-large-deg}) that given a graph $J$ on $m$ vertices, it has a subgraph $J_1$ of size $O(m \polylog m)$ such that $J \setminus J_1$ can be separated using $O(m)$ paths.
			To do so, we decompose $J$ into `robust' expanders (with total order $O(m)$) and a remainder of size $O(m \polylog m)$ and show that such expanders can be separated using a linear number of paths (in \Cref{cor:sep-dense-expander}). 

			To separate a robust expander $K$, we show that if $V$ is a set of vertices in $K$ that includes each vertex with probability $1/3$, then, with high probability, the edges of $K \setminus V$ can be separated by a linear number of paths. To find these paths, we apply the strategy given in \Cref{subsec:strategy} to $K \setminus V$, to obtain small collections $\cP$ of paths and $\cM$ of matchings that together separate the edges of $K$, while imposing a certain degree condition on the matchings in $\cM$. We then partially connect the edges of each $M \in \cM$ using a subset $V_1$ of $V$, chosen uniformly at random, thereby forming a path forest whose ends `expand well' into $V \setminus V_1$. Now, we adapt a result from \cite{bucic2022erdos} (which in turn uses ideas from a paper of Tomon \cite{tomon2022robust}) that allows for connecting pairs of vertices, disjointly, via a random set of vertices. 

		\paragraph{Step 2: separating edges touching large degrees in sparse expanders.}
			We again start by decomposing a graph $G$ on $n$ vertices and with average degree $d$ into expanders (here the expanders are not required to be robust, and so we have no uncovered edges), whose sum of orders is $O(n)$. For each such expander $H$, we show (in \Cref{lem:bound-deg}) that its edges touching high degree vertices can be separated using $O(n + \frac{e(H)}{d})$ paths. 

			Denoting by $L$ the set of large degree vertices in $H$, let us explain how to separate the edges in $H[L]$. Writing $k := n + \frac{e(H)}{d}$, we apply the strategy in \Cref{subsec:strategy}, while additionally requiring that the paths in $\cP$ and matchings in $\cM$ have size at most $d$; it is straightforward to show that there exist appropriate $\cP$ and $\cM$ of size $O(k)$. Now, given a matching $M \in \cM$, let $F$ be a path forest that extends $M$, avoids vertices in $X := \bigcup_{e \in M} V(P(e)) \setminus V(M)$, and under these conditions has as few components as possible, and as few edges as possible (prioritising the former). Suppose, towards a contradiction, that there are at least two components, and let $u$ and $v$ be leaves of distinct components of $F$. The minimality assumption allows us to show that, roughly speaking, balls around $u$ in $H \setminus X$ have few neighbours in $F$. Adapting a lemma from Liu--Montgomery \cite{liu2020solution}, it follows that more than $|H|/2$ vertices in $H$ can be reached from $u$ in $H \setminus (X \cup V(F) \setminus \{u\})$, exploiting the fact that $u$ has a large degree. Since the same holds for $v$, it follows that $F$ can be extended by a path from $u$ to $v$, a contradiction. Thus $F = F_M$ is a path that extends $M$ and avoids other edges from $\bigcup_{e \in M}P(e)$, and $\{F_M : M \in \cM\} \cup \cP$ is a collection of $O(k)$ paths separating $G[L]$.

			In fact, similar reasoning can be applied to separate edges $xy$ where $x \in L$ and $y$ has at least four neighbours in $L$, by allowing the members of $\cM$ to contain paths of length $2$ whose ends are in $L$. 

		\paragraph{Step 3: separating almost all remaining edges in $H$.}

			For each $H$ as in the previous step, let $F$ be its subgraph of `unseparated edges'; so $F$ has small maximum degree. We now decompose $F$ into `somewhat robust' expanders (with total order $O(|H|)$) and a remainder of $O(\polylog d \cdot |H|)$ edges. Consider one of these expanders $J$. If $|J|$ is fairly large with respect to $d$ (namely, if $|J| \ge 2^{(\log d)^{7})}$) then we show (in \Cref{lem:sparse-expanders}) that its edges can be separated using $O(|J|)$ edges. 

			The proof is similar to the one outlined in Step 2. Here, in addition to requiring the paths and matchings in $\cP$ and $\cM$ to have size at most $d$, we also require the edges in $\cM$ to be far from each other (the assumption that $|J|$ is large with respect to $d$ is used to prove that appropriate $\cM$ exists). Defining $M$, $X$ and $F$ as above, it suffices to show that $F$ has just one component. To this end, taking $u$ to be a leaf in $F$, we first show that the ball of radius $(\log d)^{5}$ around $u$ in $J' := J \setminus (X \cup V(F) \setminus \{u\})$ is large, exploiting the expansion property being `somewhat robust' and the fact that the other edges of $M$ are far from $u$. We then proceed like in the previous step to expand $u$ further, and reach a contradiction if $F$ has at least two components. As before, this implies the existence of a separating path system for $J$ of size $O(|J| + \frac{e(J)}{d})$.

			If instead $|J| \le 2^{(\log d)^{7}}$, we apply the first step to show that all but $O(|J| \polylog |J|)$ edges of $J$ can be separated using $O(|J|)$ paths.
			
			Altogether, we get a collection of $O(n)$ paths that separate all but $O(n \polylog d)$ edges of the graph $G$ that we started with at the beginning of the second step, concluding the iterative step.

	\subsection{Organisation of the paper} \label{subsec:structure}

		In the next section (\Cref{sec:lemmas}) we state two key lemmas --- \Cref{lem:reduce-large-deg}, which covers Step 1 above, and \Cref{lem:reduce-small-deg}, which covers Steps 2 and 3 --- and show how our main theorem (\Cref{thm:main}) follows from these lemmas.
		In \Cref{sec:decompose-expanders} we state and prove the decomposition-into-expanders lemma, \Cref{lem:decompose-expanders}, which will be used in the proofs of both key lemmas. \Cref{sec:random} contains the proof of \Cref{lem:connecting}, which allows for connecting pairs of vertices in an expander disjointly through a random set of vertices, and is a variant of a lemma from \cite{bucic2022erdos}. The first key lemma (\Cref{lem:reduce-large-deg}) is then proved in \Cref{sec:dense-expanders}. In \Cref{sec:limited-contact} we prove a variant of a lemma from \cite{liu2020solution} regarding expansion with forbidden sets. Finally, in \Cref{sec:sparse-expanders} we prove the second key lemma (\Cref{lem:reduce-small-deg}).

	\subsection{Notation} \label{subsec:notation}

		Our notation is mostly standard. Given a graph $G$, a set of vertices $X$, and integer $i \ge 0$, we write $N_G(X)$ to be the set of vertices \emph{outside} of $X$ that send at least one edge into $X$, and write $B^i_G(X)$ to be the set of vertices at distance at most $i$ from $X$ in $G$. Given a vertex $x$ and a set of vertices $X$ in a graph $G$, we denote the number of neighbours of $x$ in $X$ by $d_G(x, X)$. As is customary, we omit the subscript $G$ when it is clear from the context.
		Given a graph $G$ and two disjoint sets of vertices $A, B$ in it, we write $G[A,B]$ for the bipartite subgraph of $G$, with bipartition $(A,B)$, whose edges are the edges of $G$ with one end in $A$ and the other in $B$.
		Given vertices $x, y$, an \emph{$(x,y)$-path} is a path from $x$ to $y$. 
		We denote the length of a path $P$ by $\ell(P)$. 
		We say that a sequence of events $(E_n)_{n \ge 1}$ holds \emph{with high probability} if it holds with probability tending to $1$ as $n$ tends to infinity.
		All logarithms will be in base $2$, unless specified otherwise, and $\logstar n$ is \emph{iterated logarithm}, namely the least number of times the base-$2$ logarithms has to be applied, iteratively and starting with $n$, to yield a number which is smaller than $1$.
		When dealing with large numbers we often omit floor and ceiling signs when they are not crucial.

\section{Main lemmas} \label{sec:lemmas}

	In this section we state two key lemmas and then show how to use them to prove our main result (\Cref{thm:main}).

	The first key lemma shows that given any graph on $n$ vertices there is a linear collection of paths separating all but $O(n \polylog n)$ of its edges.

	\begin{restatable}{lemma}{lemReduceLargeDeg} \label{lem:reduce-large-deg}
		Suppose that $G$ is a graph on $n$ vertices. Then there is a subgraph $G_1 \subseteq G$, with $e(G_1) \le n (\log n)^{55}$, and a collection $\cP$ of at most $1300 n$ paths in $G$ that separates the edges of $G \setminus G_1$.
	\end{restatable}

	The second key lemma has a more complicated statement; essentially, it shows that for every graph $G$ there is a linear collection of paths that separates a subgraph of $G$ whose complement in $G$ can be decomposed into small subgraphs and a sparse remainder.

	\begin{restatable}{lemma}{lemReduceSmallDeg} \label{lem:reduce-small-deg}
		There exists $d_0$ such that the following holds.
		Suppose that $G$ is a graph on $n$ vertices with average degree $d$, where $d \ge d_0$.
		Then there is a subgraph $G_1 \subseteq G$, a collection $\HH$ of subgraphs of $G$ and a collection $\P$ of paths in $G$ with the following properties.
		\begin{enumerate}[label = \rm(\arabic*)]
			\item \label{itm:key-two-HH}
				The graphs in $\HH$ are mutually edge-disjoint; $|H| \le 2^{(\log d)^7}$ for every $H \in \HH$; and $\sum_{H \in \HH} |H| \le 4n$.
			\item \label{itm:key-two-P}
				The collection $\P$ separates the edges of $G - G_1 - \bigcup_{H \in \HH} H$, and satisfies $|\P| \le 80n$.
			\item \label{itm:key-two-G1}
				The graph $G_1$ has average degree at most $(\log d)^3$.
		\end{enumerate}
	\end{restatable}

	The following corollary combines the two key lemmas to show that given a graph with average degree $d$, there is a linear collection of paths that separates all edges but a remainder whose average degree is $O(\polylog d)$.

	\begin{corollary} \label{cor:one-step}
		Suppose that $G$ is a graph on $n$ vertices with average degree $d$. Then there is a subgraph $G_1 \subseteq G$ with average degree at most $O((\log d)^{420})$ and a collection $\P$ of $O(n)$ paths in $G$ that separates $G \setminus G_1$.
	\end{corollary}

	\begin{proof}
		We may assume that $d$ is large, because otherwise we could take $G_1$ to be the empty graph and $\P = E(G)$.
		Apply \Cref{lem:reduce-small-deg} to obtain $\HH$ and $\P$ that satisfy the conditions of the lemma.
		Namely, the graphs in $\HH$ are mutually edge-disjoint and consist of at most $2^{(\log d)^{7}}$ vertices, $\sum_{H \in \HH}|H| \le 4n$, $|\P| \le 80n$, and the paths in $\P$ separate the edges in $G'$, where $G' := G - G'' - \bigcup_{H \in \HH}H$ and $G''$ is a subgraph of $G'$ with average degree $O((\log d)^{3})$. Let $\P'$ be a collection of at most $n$ paths in $G'$ that decomposes the edges of $G'$; such a collection exists by \Cref{thm:lovasz}.

		Apply \Cref{lem:reduce-large-deg} to each graph $H$ in $\HH$; denote the resulting subgraph and collection of paths by $G_H$ and $\P_H$. So $e(G_H) \le |H|(\log |H|)^{55}$, $|\P_H| = O(|H|)$, and $\P_H$ is a collection of paths in $H$ that separates the edges of $H - G_H$. Write $G_1 := \bigcup_{H \in \HH}G_H \cup G''$. Then 
		\begin{align*}
			\sum_{H \in \HH} e(G_H)
			\le \sum_{H \in \HH} |H| (\log |H|)^{55}
			\le \left(\sum_{H \in \HH} |H| \right) \cdot (\log d)^{420}
			= O(n (\log d)^{420}).
		\end{align*}
		It follows that $e(G_1) = O(n (\log d)^{420})$, or, equivalently, $G_1$ has average degree $O((\log d)^{420})$. 
		Let $\P'_H$ be a collection of at most $|H|$ paths in $H$ that decomposes $E(H)$, and write $\Q := \P \cup \P' \cup \bigcup_{H \in \HH}(\P_H \cup \P'_H)$. Then $|Q| = O(n)$, using $\sum_{H \in \HH} |H| = O(n)$.

		We claim that $\Q$ separates the edges of $G \setminus G_1$. To see this, consider two edges $e$ and $f$ in $G \setminus G_1$. We consider four cases.
		\begin{itemize}
			\item
				There is $H \in \HH$ such that $e, f \in E(H)$.
				Then, there is $P \in \cP_H$ such that $e \in E(P)$ and $f \notin E(P)$.
			\item
				There is $H \in \HH$ such that $e \in E(H)$ and $f \notin E(H)$. Then, there is a path $P \in \cP_H'$ in $H$ that contains $e$, and so $f \notin E(P)$.
			\item
				Both $e$ and $f$ are in $G'$.
				Then, there is $P \in \cP$ such that $e \in E(P)$ and $f \notin E(P)$.
			\item
				The edge $e$ is in $G'$ and $f$ is not in $G'$. Then, there is $P \in \cP'$ such that $e \in E(P)$ and $f \notin E(P)$.
		\end{itemize}
		The subgraph $G_1$ and collection $\Q$ of paths satisfy the requirements of the corollary.
	\end{proof}

	The next corollary, applies the previous one twice, to prove a very similar statement, with the average degree of the remainder now being at most $\log d$. We could have omitted this corollary and proved \Cref{thm:main} directly from \Cref{cor:one-step}, but the counting is slightly simpler with \Cref{cor:two-steps}.

	\begin{corollary} \label{cor:two-steps}
		Let $G$ be a graph on $n$ vertices with average degree $d$, where $d$ is large. Then there is a subgraph $G_1 \subseteq G$ with average degree at most $\log d$ and a collection $\P$ of $O(n)$ paths in $G$ that separates $G \setminus G_1$.
	\end{corollary}

	\begin{proof}
		We apply \Cref{cor:one-step} twice. Denote by $G_1$ and $\P_1$ the subgraph and collection of paths resulting from the first application, and let $G_2$ and $\P_2$ be the subgraph and collection of paths resulting from a second application with the graph $G_1$. Apply \Cref{thm:lovasz} to get a collection $\P_1'$ of at most $n$ paths that decomposes $G \setminus G_1$, and define $\P_2'$ similarly with respect to $G_1 \setminus G_2$. Write $\Q = \P_1 \cup \P_2 \cup \P_1' \cup \P_2'$. Then $|Q| = O(n)$ and $Q$ separates the edges of $G \setminus G_2$ (this is easy to verify using a similar analysis to the proof of \Cref{cor:one-step}).

		By the properties of $G_1$ and $G_2$ given by \Cref{cor:one-step} we have $d(G_1) = O((\log d)^{420})$ and $d(G_2) = O((\log d(G_1))^{420}) = O((\log \log d)^{420})$. Since $d$ is large, it follows that $d(G_2) \le \log d$. The subgraph $G_2$ and collection $Q$ satisfy the requirements of the corollary.
	\end{proof}

	Finally, we deduce our main result (\Cref{thm:main}) from \Cref{cor:two-steps}.

	\begin{proof}[Proof of \Cref{thm:main}]
		Write $G_0 := G$ and $d := d(G)$. We define sequences $(G_i)_{i \ge 1}$ and collections of paths $(\P_i)_{i \ge 1}$ as follows. 
		As long as $d(G_i)$ is large enough to apply \Cref{cor:two-steps}, let $G_{i+1}$ and $\P_{i+1}$ be the graph and collection of paths resulting from an application of \Cref{cor:two-steps} to $G_i$.
		Then $d(G_{i+1}) \le \log d(G_i)$, $|\P_{i+1}| = O(n)$ and $\P$ separates the edges of $G_i \setminus G_{i+1}$.
		Suppose that the last graph we defined was $G_k$. Then $k \le \logstar n$ and $e(G_k) = O(n)$.
		For $i < k$ apply \Cref{thm:lovasz} to obtain a collection of paths $\P_i'$ that decomposes $G_i \setminus G_{i+1}$ and let $\P_k' := E(G_k)$.

		Let $\Q := \bigcup_{0 \le i < k} \P_i \cup \bigcup_{0 \le i \le k} \P_i'$.
		Then $|\Q| = O(nk) = O(n \logstar n)$. We claim that $\Q$ separates $E(G)$. To see this, let $e$ and $f$ be two edges in $G$, let $i$ be such that $e \in G_i \setminus G_{i+1}$ (with $i = k$ if $e \in G_k$), and define $j$ similarly with respect to $f$.
		If $i \neq j$ then there is a path $P \in \cP_i'$ that contains $e$ but not $f$. If $i = j \neq k$ then there is a path $P \in \cP_i$ that contains $e$ but not $f$. Finally, if $i = j = k$, then $e$ is a path in $\cP_k'$ which contains $e$ but not $f$.
		Thus $\Q$ is a collection of $O(n \logstar n)$ paths that separates the edges of $G$.
	\end{proof}

\section{Decomposition into expanders} \label{sec:decompose-expanders}

	We say that a graph $G$ on $n$ vertices is an \emph{$(\eps, s, t)$-expander} if $|N_{G - F}(X)| \ge \frac{\eps|X|}{(\log|X| + 1)^2}$ for every $X \subseteq V(G)$ and $F \subseteq E(G)$ satisfying $1 \le |X| \le 2n/3$ and $|F| \le s \cdot \min\{|X|, t\}$.
	An \emph{$(\eps, s)$-expander} is an $(\eps, s, 2n/3)$-expander (namely, $F$ is not restricted by $t$). An \emph{$\eps$-expander} is an $(\eps, 0)$-expander (namely, $F$ is always empty).
	This definition builds on previously defined notions of expansion: Koml\'os and Szemer\'edi \cite{komlos1996topological,komlos1994topological} introduced a notion of sublinear expanders, which are (up to a slight change in parameters) $\eps$-expanders in our notation.
	Haslegrave, Kim and Liu \cite{haslegrave2021extremal} introduced a notion of robust sublinear expanders, which are $(\eps, s)$-expanders in our notation. 
	We add the parameter $t$, which provides a cut-off in the robustness: vertex sets of size at most $t$ expand robustly, whereas larger vertex sets expand less robustly.

	The following lemma shows that every graph $G$ can be decomposed into expanders, or almost decomposed into robust expanders. This is a variant of Lemma 14 in \cite{bucic2022erdos} and the proof is similar but more technical. The main difference is the introduction of the parameter $t$, which allows us to obtain expanders where small sets of vertices expand robustly, while not leaving too many edges uncovered.
	As a sanity check, note that as $s$ and $t$ increase, the expansion property becomes stronger, and the potential number of uncovered edges grows.

	We will apply this lemma with several regimes for $s$ and $t$: in \Cref{sec:dense-expanders} we take $s$ to be polylogarithmic in $n$ (and take $t = 2n/3$). In \Cref{sec:sparse-expanders} we apply this lemma twice, once with $s = 0$ (in which case the choice of $t$ does not matter), and once with $s$ a large constant and $t$ being the average degree of the ground graph.

	\begin{lemma} \label{lem:decompose-expanders}
		Let $s \ge 0$, $t \ge 1$ and let $0 < \eps \le 1/48$.
		Let $G$ be a graph on $n$ vertices. Then there is a collection $\HH$ of pairwise edge-disjoint $(\eps, s, t)$-expanders that covers all but at most $48sn (\log t + 1)^2$ edges of $G$ and satisfies $\sum_{H \in \HH}|H| \le 2n$.
	\end{lemma}

	\begin{proof}
		We leave a few simple calculus claims (see \Cref{claim:calc-1,claim:calc-2,claim:calc-3}) until after this proof. Notice that we may assume $t \le \frac{2n}{3}$, by our definition of expanders.

		We prove the statement by induction on $n$, with the following slightly stronger upper bounds: $24 sn \cdot (\log t + 1)^2 (2 - \frac{1}{\log n + 1})$ on the number of edges removed, and $n \cdot (2 - \frac{1}{\log n + 1})$ on the sum of orders. 
		If $G$ is itself an $(\eps, s, t)$-expander then we may take $\HH = \{G\}$, noting that $n \cdot (2 - \frac{1}{\log n + 1}) \ge n$.

		Suppose then that $G$ is not an $(\eps, s, t)$-expander. In particular, $n \ge 2$ because the singleton graph is an $(\eps, s, t)$-expander. This means that there are subsets $X \subseteq V(G)$ and $F \subseteq E(G)$ such that $1 \le |X| \le 2n/3$, $|F| \le s \cdot \min\{|X|, t\}$, and $|N_{G - F}(X)| \le \frac{\eps|X|}{(\log|X| + 1)^2}$. Let $G_1 := G[X \cup N_{G - F}(X)]$ and let $G_2$ be the subgraph obtained from $G - X$ by removing the edges with both ends in $N_{G - F}(X)$. 
		Then $G_1$ and $G_2$ are edge-disjoint, together they cover all edges in $G - F$, and they satisfy 
		\begin{align} \label{eqn:g1g2}
			\begin{split}
				n \le |G_1| + |G_2| 
				= n + |N_{G \setminus F}(X)|
				& \le n + \frac{\eps |X|}{(\log |X| + 1)^2} \\ 
				& \le n + \frac{3\eps |G_1|}{(\log |G_1| + 1)^2}
				\le n + \frac{|G_1|}{16(\log |G_1| + 1)^2},
			\end{split}
		\end{align}
		where the penultimate inequality follows from \Cref{claim:calc-1} stated below.

		The bulk of the calculations needed for this lemma will be performed in the following claim.
		\begin{claim} \label{claim:calc-main}
			\begin{align*}
				\, |G_1|\left(2 - \frac{1}{\log|G_1| + 1}\right) + |G_2|\left(2 - \frac{1}{\log |G_2| + 1}\right)  + \frac{\min\{|G_1|, t\}}{24(\log t + 1)^2} \le  n\left(2 - \frac{1}{\log n + 1}\right).
			\end{align*}
		\end{claim}
		\begin{proof}
			Write $D$ for the difference between the left-hand side and the right-hand side.
			\begin{align*}
				D & = |G_1|\left(2 - \frac{1}{\log|G_1| + 1}\right) + |G_2|\left(2 - \frac{1}{\log |G_2| + 1}\right) - n\left(2 - \frac{1}{\log n + 1}\right) + \frac{\min\{|G_1|, t\}}{24(\log t + 1)^2}\\
				& = 2\left(|G_1| + |G_2| - n\right) - \frac{|G_1|}{\log |G_1| + 1} - \frac{|G_2|}{\log|G_2| + 1} + \frac{n}{\log n + 1} + \frac{\min\{|G_1|, t\}}{24(\log t + 1)^2}\\
				& \le \, \frac{|G_1|}{8(\log |G_1| + 1)^2} - \frac{|G_1|}{\log |G_1| + 1} - \frac{|G_2|}{\log n + 1} + \frac{n}{\log n + 1} + \frac{\min\{|G_1|, t\}}{24(\log t + 1)^2}\\
				& \le \, |G_1|\left( \frac{1}{8(\log |G_1| + 1)^2} - \frac{1}{\log |G_1| + 1} + \frac{1}{\log n + 1} + \frac{\min\{1, t/|G_1|\}}{24(\log t + 1)^2}\right),
			\end{align*}	
			where for the first inequality we used \eqref{eqn:g1g2}.
			We consider two cases: $|G_1| \ge t$ and $|G_1| < t$.
			In the first case, using \Cref{claim:calc-1}, we have 
			\begin{equation*}
				\frac{\min\{1, t/|G_1|\}}{24(\log t + 1)^2} 
				= \frac{t}{24|G_1|(\log t + 1)^2}
				\le \frac{3|G_1|}{24|G_1|(\log |G_1| + 1)^2}
				= \frac{1}{8(\log |G_1| + 1)^2}.
			\end{equation*}
			Thus, noting that $|G_1| \le |X| + \frac{\eps|X|}{(\log |X| + 1)^2} \le |X| \left(1 + \frac{1}{24}\right) \le \frac{2n}{3} \cdot \frac{25}{24} \le \frac{3n}{4}$, and using \Cref{claim:calc-3},
			\begin{align*}
				D  
				& \le \, |G_1|\left( \frac{1}{4(\log |G_1| + 1)^2} - \frac{1}{\log |G_1| + 1} + \frac{1}{\log n + 1} \right) \le 0. 
			\end{align*}
			In the second case, by \Cref{claim:calc-2}, we have $\frac{1}{8(\log|G_1|+1)^2} - \frac{1}{\log|G_1|+1} \le \frac{1}{8(\log t + 1)^2} - \frac{1}{\log t + 1}$. Thus,
			\begin{align*}
				D & \le |G_1| \left(\frac{1/8 + 1/24}{(\log t + 1)^2} - \frac{1}{\log t + 1} + \frac{1}{\log n + 1}\right) \\
				& \le |G_1| \left(\frac{1}{4(\log t + 1)^2} - \frac{1}{\log t + 1} + \frac{1}{\log n + 1}\right) \le 0,
			\end{align*}
			where the last inequality follows from \Cref{claim:calc-3}, using $1 \le t \le 2n/3$.
			Either way, $D \le 0$, as required.
		\end{proof}

		For $i \in [2]$, apply the induction hypothesis to $G_i$ to obtain a family $\HH_i$ of pairwise edge-disjoint subgraphs of $G_i$ that cover all but at most $24s(\log t + 1)^2|G_i|(2 - \frac{1}{\log |G_i| + 1})$ edges of $G_i$ and satisfy $\sum_{H \in \HH_i} |H| \le |G_i|\left(2 - \frac{1}{\log |G_i| + 1}\right)$. Take $\HH = \HH_1 \cup \HH_2$. By \Cref{claim:calc-main}, the graphs in $\HH$ cover all but at most the following number of edges in $G$:
		\begin{align*}
			& 24s(\log t + 1)^2 \left( |G_1|\left(2 - \frac{1}{\log|G_1| + 1}\right) + |G_2|\left(2 - \frac{1}{\log |G_2| + 1}\right) + \frac{\min\{|G_1|, t\}}{24(\log t + 1)^2}\right) \le \\
			& 24sn(\log t + 1)^2 \left(2 - \frac{1}{\log n + 1}\right).
		\end{align*}
		Similarly, 
		\begin{align*}
			 \sum_{H \in \HH}|H| \le 
			 \sum_{H \in \HH_1}|H| + \sum_{H \in \HH_2}|H| 
			 & \le |G_1| \left(2 - \frac{1}{\log|G_1| + 1}\right) + |G_2| \left(2 - \frac{1}{\log|G_2| + 1}\right) \\
			 & \le n \left(2 - \frac{1}{\log n + 1}\right).
		\end{align*}
		So $\HH$ satisfies the requirements, proving the induction hypothesis.
	\end{proof}

	We now state and prove the calculus claims used in the proof of \Cref{lem:decompose-expanders}.

	\begin{claim} \label{claim:calc-1}
		$\frac{x}{(\log x + 1)^2} \le \frac{3y}{(\log y+1)^2}$ for integers $x,y$ with $1 \le x \le y$.
	\end{claim}
	\begin{proof}
		Consider the function $f(x) = \frac{x}{(\log x + 1)^2}$. Then
		\begin{equation*}
			f'(x) = \frac{(\log_e(2))^2 \log_e(2x/e^2)}{(\log_e(2x))^3}.
		\end{equation*}
		In particular, $f'(x) \ge 0$ for $x \ge 4$, showing $f(x) \le f(y) \le 3f(y)$ when $4 \le x \le y$ (using $f(y) \ge 0$ for $y \ge 4$).
		A calculation shows $1/3 \le f(x) \le 1$ for $x \in [4]$. It follows that $f(x) \le 3f(y)$ for $x \in [3]$ and any integer $y \ge 1$.
	\end{proof}

	\begin{claim} \label{claim:calc-2}
		$\frac{1}{8(\log x + 1)^2} - \frac{1}{\log x + 1} \le \frac{1}{8(\log y + 1)^2} - \frac{1}{\log y + 1}$ when $1 \le x \le y$.
	\end{claim}
	\begin{proof}
		Write $f(x) = \frac{1}{8(\log x + 1)^2} - \frac{1}{\log x + 1}$. 
		Then the derivative $f'(x)$ satisfies the following for $x \ge 1$. 
		\begin{equation*}
			f'(x) = \frac{\log_e(2)(4 \log_e x + \log_e 8)}{4x(\log_e(2x))^3} \ge 0.
		\end{equation*}
		In particular, $f$ is increasing in the range $x \ge 1$, as required.
	\end{proof}

	\begin{claim} \label{claim:calc-3}
		The following holds for $1 \le x \le \log(3n/4) + 1$ and $n \ge 2$.
		\begin{equation*}
			\frac{1}{4x^2} - \frac{1}{x} + \frac{1}{\log n + 1} \le 0.
		\end{equation*}
	\end{claim}
	\begin{proof}
		Write $f(x) = \log n + 1 - 4x(\log n + 1) + 4x^2$, and notice that it suffices to prove $f(x) \le 0$ for $1 \le x \le \log(2n/3) + 1$.

		Since the set of real numbers $x$ for which $f(x) \le 0$ is an interval, it suffices to show that $f(x) \le 0$ for $x = 1$ and $x = \log(2n/3) + 1$. 
		\begin{align*}
			f(1) & = \log n + 1 - 4(\log n + 1) + 4 = 1 - 3\log n \le 0. \\ \\
			f(\log(3n/4) + 1) & = \log n + 1 - 4(\log(3n/4) + 1)(\log n + 1) + 4(\log (3n/4) + 1)^2 \\
			& = \log n + 1 - 4(\log n + \log(3/2))(\log n + 1) + 4(\log n + \log(3/2))^2 \\
			& = \log n \cdot (1 - 4 - 4\log(3/2) + 8\log(3/2)) + 1 - 4\log(3/2) + 4(\log(3/2))^2 \\
			& \le - \frac{1}{2} \cdot \log n + \frac{1}{2} \le 0.
		\end{align*}
		This completes the proof of the claim.
	\end{proof}

\section{Finding vertex-disjoint paths through a random vertex set} \label{sec:random}

	In this section we prove \Cref{lem:connecting}, which allows for connecting pairs of vertices $(x_1, y_1), \ldots, (x_r, y_r)$ in a robust expander through a random set of vertices via vertex-disjoint paths, provided that any subset of the vertices $x_1, \ldots, x_r, y_1, \ldots, y_r$ expands well. This is a variant of Theorem 10 in \cite{bucic2022erdos}, which does something very similar, but avoids the expansion assumption on $x_1, \ldots, x_r, y_1, \ldots, y_r$ and only requires the paths to be edge-disjoint.
	Our proof follows \cite{bucic2022erdos} very closely. We note that, as mentioned also in \cite{bucic2022erdos}, some of the ideas that appear in the proof are due to Tomon \cite{tomon2022robust}.

	We will introduce the relevant preliminaries (including a different notion of expansion) in \Cref{subsec:prelims}, we then prove \Cref{lem:connecting} in three steps, given in \Cref{subsec:random-step-one,subsec:expansion-step-two,subsec:expansion-step-three}.

	\subsection{Preliminaries} \label{subsec:prelims}
		Say that a graph $G$ on $n$ vertices is a \emph{weak $(\eps, s)$-expander} if $|N_{G - F}(U)| \ge \frac{\eps|U|}{(\log n)^2}$ for every $U \subseteq V(G)$ and $F \subseteq E(G)$ with $1 \le |U| \le 2n/3$ and $|F| \le s|U|$. In \cite{bucic2022erdos} such a graph is called simply an $(\eps, s)$-expander; we add `weak' to its name to distinguish it from $(\eps, s)$-expanders, which have stronger expansion properties for smaller sets of vertices.
		Note that an $(\eps, s)$-expander on at least two vertices is a weak $(\frac{\eps}{2}, s)$-expander.

		We will use the following two propositions from \cite{bucic2022erdos}, both of which show that every not-too-large vertex set $U$ in an expander either `expands well' (i.e.\ has a large neighbourhood) or `expands robustly' (i.e.\ there are many vertices in $N(U)$ with many neighbours in $U$). We state a slightly simplified version of the propositions, where there is no set of forbidden edges, because we will not use this feature.

		\begin{proposition}[Proposition 12 in \cite{bucic2022erdos}] \label{prop:strong-expansion}
			Let $\eps > 0$, let $0 < d \le s$. Suppose that $G$ is an $n$-vertex $(\eps, s)$-weak expander. Then one of the following holds, for every set $U \subseteq V(G)$ with $|U| \le \frac{2n}{3}$.
			\begin{itemize}
				\item
					$|N(U)| \ge \frac{s|U|}{2d}$, or
				\item
					$\left|\{u \in V(G) : |N(u) \cap U| \ge d\}\right| \ge \frac{\eps |U|}{(\log n)^2}$.
			\end{itemize}
		\end{proposition}

		\begin{proposition}[Proposition 13 in \cite{bucic2022erdos}] \label{prop:stars}
			Let $\eps \ge 2^{-9}$, $s \ge 8(\log n)^{13}$, and let $n$ be large.
			Suppose that $G$ is an $(\eps, s)$-weak expander,
			and let $W \subseteq V(G)$ satisfy $|W| \le 2n/3$. Then $G$ contains one of the following
			\begin{itemize}
				\item
					a collection of at least $\frac{|W|}{(\log n)^{7}}$ pairwise vertex-disjoint stars of size at least $(\log n)^9$, whose centre is in $W$ and its leaves are in $V(G) - W$,
				\item
					a bipartite graph $H$ with parts $W$ and $X \subseteq V(G) - W$, such that
					\begin{itemize}
						\item
							$|X| \ge \frac{\eps|W|}{2(\log n)^2}$,
						\item
							every vertex in $X$ has degree at least $(\log n)^4$ and every vertex in $W$ has degree at most $2(\log n)^9$ in $H$.
					\end{itemize}
			\end{itemize}
		\end{proposition}

		As in \cite{bucic2022erdos}, we also need the following hypergraph version of Hall's theorem, due to Aharoni and Haxell.
		\begin{theorem}[Aharoni--Haxell \cite{aharoni2000hall}] \label{thm:hall}
			Let $s,r \ge 1$ be integers, and let $\cH_1, \ldots, \cH_r$ be hypergraphs on the same vertex set whose edges have at most $s$ vertices. Suppose that, for every $I \subseteq [r]$, there is a matching in $\bigcup_{i \in I} \cH_i$ of size at least $s(|I|-1)$. Then there is a matching of size $r$ in $\bigcup_{i \in [r]} \cH_i$ whose $i^{\text{th}}$ edge is in $\cH_i$.
		\end{theorem}

		Additionally, we use the following martingale concentration result (see Chapter 7 in \cite{alon2016probabilistic}).
		We say that a function $f : \prod_{i = 1}^n	\Omega_i \to \bR$, where $\Omega_i$ are arbitrary sets, is \emph{$k$-Lipschitz} if $|f(u) - f(v)| \le k$ for every $u, v \in \prod_{i = 1}^n$ that differ on at most one coordinate.
		\begin{lemma} \label{lem:concentration}
			Let $X_1, \ldots, X_n$ be independent random variables, with $X_i$ taking values in a set $\Omega_i$ for $i \in [n]$, and write $X = (X_1, \ldots, X_n)$. 
			Suppose that $f : \prod_{i = 1}^n \Omega_i \to \bR$ is $k$-Lipschitz. Then,
			\begin{equation*}
				\Pr\left(|f(X) - \Ex f(X)| > t\right) \le 2 \exp \left(\frac{-2t^2}{k^2n}\right).
			\end{equation*}
		\end{lemma}

	\subsection{Expansion into a random vertex set} \label{subsec:random-step-one}

		The following lemma shows that, given a sufficiently robust weak expander $G$ and a random set of vertices $V$, obtained by including each vertex independently with probability $1/6$, any set $U$ expands well in $V$, while avoiding a given small set of vertices $Z$, with quite large probability. This lemma is a variant of Lemma 17 in \cite{bucic2022erdos}, where instead of a forbidden set of edges we have a forbidden set of vertices, and the proof is essentially the same.

		\begin{lemma} \label{lem:expansion-random}
			Let $2^{-9} \le \eps < 1$, $s \ge 8(\log n)^{13}$ and let $n$ be large. Suppose that $G$ is an $n$-vertex weak $(\eps, s)$-expander and let $U, Z \subseteq V(G)$ be sets satisfying $|U| \ge (\log n)^{23}$ and $|Z| \le \frac{|U|}{(\log n)^{3}}$. Let $V$ be a random subset of $V(G)$, obtained by including each vertex independently with probability $1/6$.
			Then, with probability at least $1 - \exp\left(-\Omega\left(\frac{|U|}{(\log n)^{22}}\right)\right)$,
			\begin{equation*}
				\left| B^{(\log n)^{4}}_{G[V']}(U \cap V') \right| > \frac{|V|}{2},
			\end{equation*}
			where $V' := V \setminus Z$.
		\end{lemma}

		\begin{proof}
			Take $\ell := (\log n)^4$, and let $p$ satisfy $1 - (1 - p)^{\ell}(1 - \frac{3}{20})(1 - \frac{1}{120}) = \frac{1}{6}$, so that
			\begin{align*}
				1 - p\ell 
				\le (1 - p)^{\ell} 
				= \frac{1 - 1/6}{(1 - 3/20) \cdot (1 - 1/120)}
				\le \frac{99}{100},
			\end{align*}
			which implies that $p \ge \frac{1}{100 (\log n)^{4}}$.
			Let $V_1, \ldots, V_{\ell}, V^*, V^{**}$ be random sets, where $V_i$ is obtained by including each vertex with probability $p$, independently, $V^*$ is obtained by including each vertex with probability $\frac{1}{120}$, independently, and $V^{**}$ is obtained by including each vertex with probability $\frac{3}{20}$, independently. Notice that each vertex is in $V_1 \cup \ldots \cup V_{\ell} \cup V^* \cup V^{**}$ with probability $1/6$, so we may think of $V$ as the union $V_1 \cup \ldots \cup V_{\ell} \cup V^* \cup V^{**}$.

			Define $U^* = (U \cap V^*) \setminus Z$. Then, with probability $1 - \exp\left(-\Omega\left(|U|\right)\right)$, we have $|U^*| \ge \frac{1}{200}|U| - |Z| \ge \frac{1}{400}|U|$.
			Define $B_0 := U^*$ and, for $i \ge 1$, let $B_i$ be the set of vertices in $G$ that can be reached by a path in $G \setminus Z$ that starts in $U^*$, has length at most $i$, and its interior is in $V_1 \cup \ldots \cup V_i$. We emphasise that $B_i$ is only required to be disjoint from $Z$, and need not be a subset of $V_i$. Notice that $B_i \subseteq B_{i+1}$ for every $i \ge 0$, implying that $|B_i| \ge |B_0| = |U^*| \ge \frac{1}{400}|U|$ for $i \ge 0$ when $|U^*| \ge \frac{1}{400}|U|$.

			\begin{claim}\label{claim:Bi-expand}
				If $|B_i| \le \frac{2}{3}n$, then, with probability at least $1 - \exp\left(-\Omega\left(\frac{|U|}{(\log n)^{22}}\right)\right)$,
				\begin{equation*}
					\left|B_{i+1} \setminus B_i\right| \ge \frac{\eps |B_i|}{10^4(\log n)^2}.
				\end{equation*}
			\end{claim}
			\begin{proof}
				Notice that a vertex in $N(B_i) \setminus Z$ is in $B_{i+1} \setminus B_i$ if at least one of its neighbours in $B_i$ is sampled into $V_{i+1}$.
				We consider the two possible outcomes of \Cref{prop:stars} for $B_i$. 

				Suppose that the first outcome holds, so there are $\frac{|B_i|}{(\log n)^7}$ pairwise vertex-disjoint stars of size $(\log n)^9$ with centres in $B_i$ and leaves in $N(B_i)$. 
				By a Chernoff bound, with probability at least $1 - \exp\left(-\Omega\left(\frac{p|B_i|}{(\log n)^7}\right)\right) = 1 - \exp\left(-\Omega\left(\frac{|U|}{(\log n)^{11}}\right)\right)$, at least $\frac{p|B_i|}{2(\log n)^7}$ centres are added to $V_{i+1}$, implying that
				\begin{equation*}
					\left|B_{i+1} \setminus B_i\right| 
					\ge (\log n)^9 \cdot \frac{p|B_i|}{2(\log n)^7} - |Z|
					\ge \frac{|B_i|}{200(\log n)^2} - \frac{|U|}{(\log n)^3}
					\ge \frac{|B_i|}{400(\log n)^2}.
				\end{equation*}

				Now suppose that the second outcome holds, so there is a bipartite subgraph $H \subseteq G$ with parts $B_i$ and $X \subseteq V(G) \setminus B_i$, with $|X| \ge \frac{\eps |B_i|}{4(\log n)^2}$, such that vertices in $X$ have degree at least $d := (\log n)^4$ in $H$ while vertices in $B_i$ have degree at most $D := 2(\log n)^9$ in $H$.
				Let $Y$ be the set of vertices in $X$ that do not have an $H$-neighbour in $V_{i+1}$. Note that $\Ex |Y| \le |X| (1 - p)^d \le |X| e^{-pd} \le \frac{399}{400}|X|$. Note also that $|Y|$ is $D$-Lipschitz, since the outcome of the sampling of any single vertex in $B_i$ affects the outcome of at most $D$ vertices in $X$. Thus, by \Cref{lem:concentration},
				\begin{align*}
					\Pr\left(|Y| > \frac{799}{800}|X|\right) \le \Pr\left(|Y| > \Ex |Y| + \frac{|X|}{800}\right)
					& \le 2\exp\left(-\frac{|X|^2}{2 \cdot 800^2 \cdot D^2|B_i|}\right) \\
					& = \exp\left(-\Omega\left(\frac{|U|}{(\log n)^{22}}\right)\right).
				\end{align*}
				So, with probability at least $1 - \exp\left(-\Omega\left(\frac{|U|}{(\log n)^{22}}\right)\right)$, we have $|B_{i+1} \setminus B_i| \ge \frac{|X|}{800} - |Z| \ge \frac{\eps|B_i|}{10^4 (\log n)^2}$.
			\end{proof}

			Iterating the claim, with probability at least 
			\begin{equation*}
				1 - (\log n)^4 \cdot \exp\left(-\Omega\left(\frac{|U|}{(\log n)^{22}}\right)\right)
				\ge 1 - \exp\left(-\Omega\left(\frac{|U|}{(\log n)^{22}}\right)\right),
			\end{equation*}
			if $|B_i| \le \frac{2}{3}n$ and $i \le (\log n)^4$, then,
			\begin{equation*}
				|B_i| \ge \left(1 + \frac{\eps}{10^4 (\log n)^2}\right)^i |U|,
			\end{equation*}
			which implies $|B_{\ell}| \ge \frac{2}{3}n$ with probability $1 - \exp\left(-\Omega\left(\frac{|U|}{(\log n)^{22}}\right)\right)$.

			To complete the proof, note that any vertex in $B_{\ell}$ that gets sampled into $V^{**}$ (or that is already in $V_1 \cup \ldots \cup V_{\ell} \cup V^*$) is in the set $B' := B_{G[V']}^{(\log n)^4}(U \cap V')$. By Chernoff, with probability at least $1 - e^{-\Omega(n)}$, at least $\frac{99}{100}\cdot \frac{3}{20} \cdot \frac{2}{3}n = \frac{99}{1000}n$ vertices of $B_{\ell}$ are in $V$, and $|V| \le \frac{101}{100} \cdot \frac{1}{6} n = \frac{101}{600}n$. If all this holds, then $|B'| > \frac{1}{2}|V|$, as required.
		\end{proof}

		The next corollary, which is a variant of Lemma 19 in \cite{bucic2022erdos}, boosts the probability of expansion of $U$ into $V$, so as to be amenable to an application of the union bound, with the expense of a slightly smaller upper bound on $|Z|$.

		\begin{corollary}\label{cor:expansion-random}
			Let $2^{-9} \le \eps < 1$, $s \ge 2^{11}(\log n)^{50}$, and let $n$ be large. Suppose that $G$ is an $n$-vertex weak $(\eps, s)$-expander, and let $V$ be a random subset of $V(G)$, obtained by including each vertex independently with probability at least $1/6$. Then, with high probability, if $U, Z \subseteq V(G)$ are sets with $|Z| \le \frac{|U|}{(\log n)^{30}}$, then 
			\begin{equation} \label{eqn:NU-expand}
				\left|B^{(\log n)^4}_{G[V']}\left(N(U) \cap V'\right)\right| > \frac{|V|}{2},
			\end{equation}
			where $V' := V \setminus Z$.
		\end{corollary}

		\begin{proof}
			Say that a set $U \subseteq V(G)$ \emph{expands well} if $|N(U)| \ge |U|(\log n)^{24}$. Given a (non-empty) set $U$ which expands well, and a set $Z$ with $|Z| \le |U|$, \Cref{lem:expansion-random} (applied with $U_{\ref{lem:expansion-random}} = N(U)$) tells us that \eqref{eqn:NU-expand} holds, with probability at least $1 - \exp\left(-\Omega\left(|U|(\log n)^{2}\right)\right)$.
			 
			By a union bound, we conclude that the probability that \eqref{eqn:NU-expand} fails for $(U, Z)$, with $U$ well-expanding and $|Z| \le |U|$, is at most
			\begin{align*}
				\sum_{u = 1}^n \binom{n}{u}^2 \exp\left(-\Omega\left(u(\log n)^{2}\right)\right)
				& \le \sum_{u = 1}^n \exp\left(2u \log n - \Omega\left(u(\log n)^2\right)\right) \\
				& = \sum_{u = 1}^n \exp\left(-\Omega\left(u (\log n)^2\right)\right) = o(1/n).
			\end{align*}
			For the rest of the proof we assume that \eqref{eqn:NU-expand} holds for all such pairs $(U, Z)$. 
			To complete the proof, we will deduce that \eqref{eqn:NU-expand} holds for all pairs $(U, Z)$, where $U$ need not expand well, $|U| \le \frac{2}{3}n$, and $|Z| \le \frac{|U|}{(\log n)^{30}}$. Fix such $U, Z$.

			Write $d := \frac{2}{\eps}(\log n)^{26}$. By \Cref{prop:strong-expansion}, either $|N(U)| \ge \frac{s|U|}{2d}$, or $\left|\{u : |N(u) \cap U| \ge d \} \right| \ge \frac{\eps |U|}{(\log n)^2}$. Notice that the first outcome implies that $U$ expands well and so we already know that \eqref{eqn:NU-expand} holds. We thus assume that the second one holds, and write $W := \{u : |N(u) \cap U| \ge d\}$. Let $U'$ be a subset of $U$ of size $\frac{|U|}{d}$, chosen uniformly at random. For a fixed $w \in W$, the probability that $w$ has no neighbours in $U'$ is at most
			\begin{equation*}
				\frac{\binom{|U|-d}{|U|/d}}{\binom{|U|}{|U|/d}} \le \left(\frac{|U|-d}{|U|}\right)^{|U|/d} \le e^{-1}.
			\end{equation*}
			It follows that $\Ex[|W \cap N(U')|] \ge (1 - e^{-1})|W| \ge \frac{\eps|U|}{2(\log n)^2} = |U'|(\log n)^{24}$. In particular, there is a subset $U' \subseteq U$ of size $\frac{|U|}{d}$ with $|N(U')| \ge |U'|(\log n)^{24}$. Since such $U'$ expands well, and $|Z| \le \frac{|U|}{(\log n)^{30}} \le |U'|$, Equation \eqref{eqn:NU-expand} holds for $(U', Z)$ and thus for $(U, Z)$.
		\end{proof}	
	
	\subsection{A path connection through a random set} \label{subsec:expansion-step-two}

		Our end goal is to be able to join (disjointly) $r$ pairs of vertices through a random set $V$, provided that they expand well. Here we show how to join one pair of vertices (while avoiding a small set of forbidden vertices). This is a variant of Proposition 8 in \cite{bucic2022erdos}, where in addition to forbidding vertices instead of edges, we also impose the expansion property on the vertices we wish to connect. The proof is essentially the same. 

		\begin{lemma} \label{lem:hall-condition}
			Let $2^{-9} \le \eps < 1$, $s \ge 2^{11}(\log n)^{50}$, and let $n$ be large. Suppose that $G$ is an $n$-vertex weak $(\eps, s)$-expander, and let $V$ be a random subset of $V(G)$, obtained by including each vertex independently with probability $1/6$.
			Then, with high probability, the following holds for every $r$: if $x_1, \ldots, x_r, y_1, \ldots, y_r$ are distinct vertices, satisfying $|N(X)| \ge |X|(\log n)^{50}$ for every subset $X \subseteq \{x_1, \ldots, x_r, y_1, \ldots, y_r\}$, and $Z$ is a set of size at most $2r(\log n)^{12}$ which is disjoint from $\{x_1, \ldots, x_r, y_1, \ldots, y_r\}$, then for some $i \in [r]$ there is an $(x_i, y_i)$-path in $G$ whose interior is in $V \setminus Z$ and whose length is at most $(\log n)^6$.
		\end{lemma}		

		\begin{proof}
			Fix an outcome of $V$ such that the following holds, where $V' = V \setminus Z$. 
			\begin{itemize}
				\item
					Equation \eqref{eqn:NU-expand} holds for every disjoint $U, Z \subseteq V(G)$ with $|Z| \le \frac{|U|}{(\log n)^{30}}$,
				\item
					$|V| \ge \frac{n}{8}$,
				\item
					$|N(X) \cap V'| \ge |X| (\log n)^{49}$ for every $X \subseteq \{x_1, \ldots, x_r, y_1, \ldots, y_r\}$ of size at least $r/2$.
			\end{itemize}
			By \Cref{cor:expansion-random}, a Chernoff bound and the assumption on $\{x_1, \ldots, x_r, y_1, \ldots, y_r\}$, these three assumptions hold with high probability, so we are justified in making them.
			Write $\ell := (\log n)^4+1$.
			For a set of vertices $X$ and integer $d \ge 1$, define $R^d(X) := B^d_{G[V']}(N(X) \cap V')$.

			Let $X_1$ be the set of vertices $x$ in $\{x_1, \ldots, x_r\}$ satisfying $|R^{\ell \log n}(x)| \le \frac{|V|}{2}$. 
			\begin{claim}
				$|X_1| < \frac{r}{2}$.
			\end{claim}
			\begin{proof}
				Suppose $|X_1| \ge \frac{r}{2}$. We will show that there is a sequence $(X_i)_{i \ge 1}$, such that $X_{i+1} \subseteq X_{i}$, $|X_{i+1}| \le \max\{1, \frac{|X_i|}{2}\}$, and 
				\begin{equation} \label{eqn:R}
					\left|R^{i\ell}(X_i)\right| > \frac{|V|}{2}, 
				\end{equation}
				for $i \ge 1$. 

				Writing $U := N(X_1) \cap V'$, notice that $|U| \ge \frac{r}{2}(\log n)^{49}$, using $|X_1| \ge \frac{r}{2}$ and the third item above. Thus $|U| \ge |Z| (\log n)^{30}$, and, by \eqref{eqn:NU-expand}, 
				\begin{equation*}
					\left|B^{(\log n)^4}_{G[V']}\big(N(U) \cap V'\big)\right| > \frac{|V|}{2}. 
				\end{equation*}
				As $B^{(\log n)^4}_{G[V']}\big(N(U) \cap V'\big) \subseteq R^{\ell}(X_1)$, it follows that $|R^{\ell}(X_1)| > \frac{|V|}{2}$, proving \eqref{eqn:R} for $i = 1$.	

				Now suppose that $X_1 \supseteq \ldots \supseteq X_j$ and \eqref{eqn:R} holds for $i \in [j]$. If $|X_i| = 1$ we take $X_{i+1} = X_i$ (which clearly satisfies the requirements). Otherwise, by dividing $X_i$ into at most three sets of size at most $\frac{|X_i|}{2}$, there is a subset $X_{i+1} \subseteq X_i$ of size at most $\frac{|X_i|}{2}$ satisfying $|R^{i\ell}(X_{i+1})| \ge \frac{|V|}{6}$. Consider the set $U := R^{i\ell}(X_{i+1})$. It satisfies $|U| \ge \frac{|V|}{6} \ge \frac{n}{48} \ge \frac{r}{48}(\log n)^{50} \ge |Z|(\log n)^{30}$, where the third inequality follows implicitly from the assumption that $|N(\{x_1, \ldots, x_r\})| \ge r(\log n)^{50}$. Thus, by \eqref{eqn:NU-expand},
				\begin{equation*}
					\left|B_{G[V']}^{(\log n)^4}\left(N(U) \cap V'\right)\right|
					> \frac{|V|}{2}.
				\end{equation*}
				Notice that
				\begin{equation*}
					B^{(\log n)^4}_{G[V']}\big(N(U) \cap V'\big) 
					\subseteq B_{G[V']}^{\ell}(U)
					= R^{(i+1) \ell}(X_{i+1}),
				\end{equation*}
				so $X_{i+1}$ has the required properties.
				This completes the proof of the existence of a sequence $(X_i)_{i \ge 1}$ with the above properties. 

				Fix such a sequence, and take $i := \log n$. Then $|X_i| \le \max\{1, 2^{-\log n}|X_1|\} = 1$. This means $|R^{\ell \log n}(x)| > \frac{|V|}{2}$ for the single vertex $x$ in $X_i$, contradicting the choice of $X_1$.
			\end{proof}
			Take $Y_1$ to be the set of vertices $y \in \{y_1, \ldots, y_r\}$ with $|R^{\ell \log n}(y)| \le \frac{|V|}{2}$. Then, analogously to the above claim, $|Y_1| < \frac{r}{2}$.
			Hence, there exists $i \in [r]$ such that $|R^{\ell \log n}(x_i)|, |R^{\ell \log n}(y_i)| > \frac{|V|}{2}$, showing that there is an $(x,y)$-path of length at most $2\ell \log n + 2$, with interior in $V'$, as required.
		\end{proof}
	
	\subsection{Many disjoint path connections through a random set} \label{subsec:expansion-step-three}

		Finally, we prove that given a sufficiently robust weak expander and pairs $(x_1, y_1), \ldots, (x_r, y_r)$, such that the set $\{x_1, \ldots, x_r, y_1, \ldots, y_r\}$ expands well, the pairs can be connected via vertex-disjoint paths through a random set $V$. This is a variant of Lemma 10 in \cite{bucic2022erdos} (though their statement looks quite different).

		\begin{lemma} \label{lem:connecting}
			Let $2^{-9} \le \eps < 1$, $s \ge 2^{11}(\log n)^{50}$, and let $n$ be large. Suppose that $G$ is an $n$-vertex weak $(\eps, s)$-expander, and let $V$ be a random subset of $V(G)$, that includes each vertex independently with probability $\frac{1}{6}$.
			Then, with high probability, for every sequence of distinct vertices $x_1, \ldots, x_r, y_1, \ldots, y_r$ satisfying $|N(X)| \ge |X|(\log n)^{50}$ for every $X \subseteq \{x_1, \ldots, x_r, y_1, \ldots, y_r\}$, there is a sequence of paths $P_1, \ldots, P_r$, whose interiors are pairwise vertex-disjoint and in $V$, such that $P_i$ is a path from $x_i$ to $y_i$.
		\end{lemma}

		\begin{proof}
			We assume that the conclusion of \Cref{lem:hall-condition} holds (it holds with high probability, so we are justified in making this assumption).
			Let $\cH_i$ be the hypergraph on vertex set $V$ with edges the interiors of $P$, for all $(x_i, y_i)$-paths $P$ of length at most $\ell := (\log n)^6$ with interior in $V$.
			
			We will use \Cref{thm:hall}. Fix a subset $I \subseteq [r]$. We wish to show that there is a matching of size $\ell(|I|-1)+1$ in $\cH' := \bigcup_{i \in I} \cH_i$. Suppose no such matching exist, and let $\cM'$ be a maximal matching in $\cH'$. Then $|\cM'| \le \ell(|I|-1)$ and every edge in $\cH'$ intersects $Z := V(\cM')$, a set of size at most $\ell^2(|I|-1) \le |I|(\log n)^{12}$. The conclusion of \Cref{lem:hall-condition} tells us that for some $i \in I$ there is an $(x_i, y_i)$-path $P$ of length at most $\ell$ whose interior is in $V \setminus Z$. But this means that $V(P) \setminus \{x_i, y_i\}$ is an edge in $\cH'$ that does not intersect $Z$, a contradiction.

			So, the assumptions in \Cref{thm:hall} are satisfied, showing that there is a matching $\cM$ of size $r$ in $\bigcup_{i \in [r]}\cH_i$, whose $i$-th edge is in $\cH_i$. Let $P_i$ be the path corresponding to the $i$-th edge in $\cM$. Then $P_i$ is an $(x_i, y_i)$-path with interior in $V$, and the interiors of the $P_i$'s are pairwise vertex-disjoint, proving the lemma.
		\end{proof}

\section{Proof of \Cref{lem:reduce-large-deg}: separating dense expanders} \label{sec:dense-expanders}

	In this section we prove the first key lemma, \Cref{lem:reduce-large-deg}, which shows that for any graph $G$, there is a linear (i.e.\ of size $O(|G|)$) path system that separates a subgraph of $G$, whose complement in $G$ has polylogarithmic average degree. 
	This will follow quite easily from \Cref{lem:decompose-expanders}, that decomposes graphs into expanders, and \Cref{cor:sep-dense-expander} that shows that robust weak expanders can be separated by a linear path system. 

	To prove the latter corollary, given a robust weak expander $G$, we will take a random partition $\{V_1, V_2, V_3\}$ of $V(G)$ and show that, with high probability, $G - V_i$ can be separated using a linear path system, for $i \in [3]$. Taking the union of these collections of paths, as well as a decomposition of $G - V_i$ into a linear number of paths, we obtain the desired linear separating paths system for $G$. 

	It thus suffices to prove that, with high probability, if $V$ is a random set, obtained by including each vertex independently with probability $1/3$, then $G \setminus V$ has a linear separating path system. This is done in the following lemma, whose proof proceeds as follows.
	First, we decompose $G - V$ into at most $n$ paths $\cP$ (using \Cref{thm:lovasz}), and then we show that $G - V$ can be further decomposed into $O(n)$ matchings $M$, each of which intersects each path in $\cP$ in at most one edge, and moreover each matching satisfies a certain degree condition.
	Next, we take a random partition $\{V_1, V_2\}$ of $V$, taken uniformly at random. We first extend each $M$ through $V_1$ to a path forest $F$, in such a way that any set of leaves of $F$ expands in $G$, and thus in $V$. This is exactly the property needed for an application of \Cref{lem:connecting}, so we can use it to connect the leaves of $F$ via paths through $V$, vertex-disjointly. We thus get a path $P_M$ that extends $M$ and contains no other edges from $G \setminus V$, and so $\cP \cup \{P_M : M \in \cM\}$ is a linear collection of paths separating $G \setminus V$.

	\begin{lemma} \label{lem:sep-random-set}
		Let $2^{-9} \le \eps < 1$, $s \ge 2^{11} (\log n)^{51}$, and let $n$ be large.
		Suppose that $G$ is a weak $(\eps, s)$-expander, and let $V$ be a random subset of $V(G)$, that includes each vertex independently with probability $\frac{1}{3}$. Then, with high probability, there is a collection of $211n$ paths that separates $G \setminus V$. 
	\end{lemma}
	
	\begin{proof}
		Let $\{V_1, V_2\}$ be a partition of $V$, chosen uniformly at random. We will make the following assumptions.
		\begin{enumerate}[label = \rm(\alph*)]
			\item \label{itm:V1}
				if $|N(U)| \ge |U|(\log n)^2$ then $|N(U) \cap V_1| \ge \frac{1}{10}|N(U)|$, for every set of vertices $U$,
			\item \label{itm:V2}
				the conclusion of \Cref{lem:connecting} holds for $V_2$.
		\end{enumerate}
		Noting that each vertex is included in $V_i$ with probability $\frac{1}{6}$, independently, a Chernoff bound and \Cref{lem:connecting} allow us to make these assumptions.

		Write $G' := G \setminus V$.
		Let $\cP$ be a collection of at most $n$ paths that decomposes $G'$. For an edge $e = xy$ in $G'$, let $P(e)$ be the unique path in $\cP$ that contains $e$, and define $\bd(e) := \min\{d_G(x), d_G(y)\}$.
		Notice that, by expansion, $\delta(G) \ge s \ge (\log n)^{51}$ and so $\bd(e) \ge (\log n)^{51}$ for every edge $e$ in $G$.

		\begin{claim}
			There is a collection $\cM$ of $210n$ pairwise edge-disjoint matchings in $G'$ that covers all edges in $G'$ and every $M \in \cM$ satisfies: every path in $\cP$ intersects at most one edge in $M$; and for every $d$ there are at most $\frac{d}{50}$ edges $e$ in $M$ with $\bd(e) \le d$.
		\end{claim}

		\begin{proof}
			Let $\cM$ be a collection of $210n$ pairwise edge-disjoint matchings such that every $M \in \cM$ satisfies the following properties,
			\begin{itemize}
				\item
					every path in $\cP$ contains at most one edge from $M$,
				\item
					for every integer $r \ge 1$, there are at most $\frac{1}{200}2^r$ edges $e$ in $M$ with $\bd(e) \in [2^{r-1}, 2^r)$,
			\end{itemize}
			and $\cM$ covers a maximum number of edges in $G'$ under these conditions.
			We will show that the matchings in $\cM$ together cover all edges of $G'$.
			Suppose not, and take $e$ to be an edge in $G'$ such that $e \notin \bigcup \cM$. Let $r$ be the integer for which $\bd(e) \in [2^{r-1}, 2^r)$.

			There are at most $n$ matchings in $\cM$ containing an edge of $P(e)$, and at most $2n$ matchings containing an edge that intersects $e$. 
			Notice that $G$ has at most $2^rn$ edges touching vertices of degree at most $2^r$. In particular, $G'$ has at most $2^r n$ edges $f$ with $\bd(f) \in [2^{r-1}, 2^r)$, so the number of matchings $M \in \cM$ that contain at least $\frac{1}{200}2^r - 1$ such edges is at most
			\begin{equation*}
				\frac{2^r n}{\frac{1}{200}2^r - 1}
				= 200 n + \frac{40000 n}{2^r - 200} \le 201 n, 
			\end{equation*}
			using that $r \ge \log(\bd(e)) - 1 \ge 51 \log \log n - 1$ and that $n$ is large. 
			Hence, there exists $M \in \cM$ that has no edges in $P(e)$ or edges that intersect $e$, and has at most $\frac{1}{200}2^r - 1$ edges $f$ with $\bd(f) \in [2^{r-1}, 2^r)$. Replacing $M$ with $M \cup \{e\}$, we reach a contradiction to the maximality of $\cM$, showing that $\cM$ covers all edges in $G'$.

			It remains to show that there are at most $\frac{d}{50}$ edges $e$ in $M$ with $\bd(e) \le d$, for every $d \ge 1$ and $M \in \cM$.
			To see this, fix $M$ and $d$, and let $r$ be the integer such that $d \in [2^{r-1}, 2^r)$. Then, there are at most $\frac{1}{200}(1 + 2 + \ldots + 2^r) \le \frac{1}{100}2^r \le \frac{d}{50}$ edges $e$ in $M$ with $\bd(e) \le 2^r$, as required.
		\end{proof}
		Let $\cM$ be as in the above claim. For an edge $e$ in $G'$, denote by $M(e)$ the unique matching in $\cM$ containing $e$.
		Fix $M \in \cM$. We will show that there is a path $P_M$ that contains $E(M)$ and its edges which are not in $M$ have at least one end in $V = V_1 \cup V_2$.
		Write $M = \{e_1, \ldots, e_{|M|}\}$, where $\bd(e_1) \le \ldots \le \bd(e_{|M|})$.
		We will define a sequence $F_0, \ldots, F_{|M|}$ of path forests, as follows. Take $F_0 := M$. Having defined $F_{i-1}$, if a vertex $v$ in $e_i$ is a leaf in $F_{i-1}$ and has a common neighbour $u$ in $V_1$ with a leaf $v'$ of another component in $F_{i-1}$, add the edges $uv, uv'$ to $F_{i-1}$ (thereby obtaining a new path forest which has one component fewer). Repeat this with the other vertex in $e_i$, and denote the resulting graph by $F_i$. 
		Write $F := F_{|M|}$, denote the components of $F$ by $P_1, \ldots, P_r$, and let $x_i, y_i$ be the ends of $P_i$.
		\begin{claim}
			Every $X \subseteq \{x_1, \ldots, x_r\}$ satisfies $|N(X)| \ge 2|X| (\log n)^{50}$.
		\end{claim}
		\begin{proof}
			Let $i_j$ be such that $x_j \in e_{i_j}$, for $j \in [r]$, and suppose that $i_1 < \ldots < i_r$.
			We claim that $\left|N(x_j) \cap (N(x_{j+1}) \cup \ldots \cup N(x_r)) \cap V_1\right| \le 2 i_j$, for $j \in [r]$.
			Indeed, the forest $F_{i_j}$ uses at most $2i_j$ vertices from $V_1$. Because $x_j$ remains a leaf in $F_{i_j}$, the vertices $x_j$ and $x_{\ell}$, for $\ell \in [j+1, r]$, have no common neighbours in $V_1 \setminus V(F_{i_j})$, showing that $x_j$ has at most $2i_j$ neighbours in $V_1$ that are also neighbours of a vertex in $\{x_{j+1}, \ldots, x_r\}$, as claimed.

			We additionally claim that, if $v$ is a leaf of $F$ and a vertex in $e_i$, then $|N(v) \cap V_1| \ge 2i + 2(\log n)^{50}$.
			Recall that $M$ has at most $\frac{d}{50}$ edges $e$ with $\bd(e) \le d$, for every $d$. Since there are at least $i$ edges $e$ in $M$ with $\bd(e) \le \bd(e_i)$, it follows that $\bd(e_i) \ge 50i$. Recalling Assumption \ref{itm:V1} and that $\bd(e_i) \ge (\log n)^{51}$, we find that $|N(v) \cap V_1| \ge \frac{1}{10} |N(v)| \ge \frac{1}{10}\bd(e_i) \ge \max\{5i, (\log n)^{51}\} \ge 2i + 2(\log n)^{50}$.

			Writing $X = \{x_{i_1}, \ldots, x_{i_{\rho}}\}$, where $i_1 < \ldots < i_{\rho}$, we have
			\begin{align*}
				|N(X)| \ge \left|N(X) \cap V_1\right| 
				& \ge \sum_{j \in [\rho]} \left(|N(x_{i_j}) \cap V_1| - \left|N(x_{i_j}) \cap \left(N(x_{i_{j+1}}) \cup \ldots \cup N(x_{i_\rho})\right) \cap V_1\right|\right) \\
				& \ge \sum_{j \in [\rho]} \left(2i_j + 2(\log n)^{50} - 2i_j\right) 
				= \rho \cdot 2(\log n)^{50} = 2|X| (\log n)^{50}, 
			\end{align*}
			proving the claim.
		\end{proof}
		An analogous argument shows that $|N(Y)| \ge 2|Y| (\log n)^{50}$ for every $Y \subseteq \{y_1, \ldots, y_r\}$. Thus, if $X \subseteq \{x_1, \ldots, x_r, y_1, \ldots, y_r\}$, without loss of generality the set $X' := X \cap \{x_1, \ldots, x_r\}$ has size at least $\frac{1}{2}|X|$, and then by the claim $|N(X)| \ge |N(X')| \ge 2|X'| (\log n)^{50} \ge |X| (\log n)^{50}$. Assumption \ref{itm:V2} thus shows that there exist paths $Q_1, \ldots, Q_{r-1}$, where $Q_i$ has ends $y_i, x_{i+1}$, and the interiors of the $Q_i$'s are pairwise vertex-disjoint and contained in $V_2$. Because $F$ was chosen so as not to contain vertices from $V_2$, the union $Q_1 \cup \ldots \cup Q_{r-1} \cup F$ is a path $P_M$ that contains $M$ and no other edges from $G'$.

		Take $\cQ := \{P_M : M \in \cM\}$. We claim that $\cP \cup \cQ$ separate the edges of $G'$.
		Indeed, given edges $e, f$ in $G'$, it is easy to see that one of $P(e)$ and $P_{M(e)}$ contains $e$ but not $f$.

		To summarise $\cP \cup \cQ$ is a collection of at most $211n$ paths that separates the edges $e$ of $G'$, as required.
	\end{proof}

	The following easy corollary of \Cref{lem:sep-random-set} shows that weak robust expanders have linear separating path systems.

	\begin{corollary} \label{cor:sep-dense-expander}
		Let $2^{-9} \le \eps < 1$ and $s \ge (\log n)^{51}$.
		Suppose that $G$ is a weak $(\eps, s)$-expander. Then there is a collection $\cP$ of at most $636 n$ paths that separates the edges in $G$.
	\end{corollary}

	\begin{proof}
		Note that we may assume that $n$ is large (otherwise take $G_1 = G$ and $\cP = \emptyset$).
		Let $\{V_1, V_2, V_3\}$ be a random partition of $V(G)$, chosen uniformly at random.
		By \Cref{lem:sep-random-set}, with high probability, there is a collection of paths $\cP_i$ of size at most $211n$ that separates $G[V_{i+1} \cup V_{i+2}]$ (indices taken modulo $3$); fix an outcome of $V_1, V_2, V_3$ that has these properties.
		Let $\cQ_i$ be a collection of at most $n$ paths that decomposes $G[V_{i+1} \cup V_{i+2}]$, for $i \in [3]$.

		We claim that $\cP := \cP_1 \cup \cP_2 \cup \cP_3 \cup \cQ_1 \cup \cQ_2 \cup \cQ_3$ separates the edges of $G$. 
		Indeed, consider two edges $e, f$ in $G$, and pick $i \in [3]$ so that $e \in G[V_{i+1} \cup V_{i+2}]$ (notice that such $i$ exists but need not be unique).
		If $f \in G[V_{i+1} \cup V_{i+2}]$ then there is a path in $\cP_i$ that contains $e$ but not $f$, and otherwise there is a path in $\cQ_i$ that contains $e$ but not $f$.

		To complete the proof it suffices to observe that $|\cP| \le 3 \cdot 211 n + 3n \le 636n$.
	\end{proof}

	Finally, we prove the first key lemma (\Cref{lem:reduce-large-deg}), restated here.
	\lemReduceLargeDeg*

	\begin{proof}[Proof of \Cref{lem:reduce-large-deg}]
		We assume that $n$ is large (otherwise, we can take $G_1 = G$ and $\cP = \emptyset$).
		Take $s = (\log n)^{51}$, $\eps = 2^{-8}$ and $t = \frac{2n}{3}$.
		Apply \Cref{lem:decompose-expanders} to obtain a collection $\cH$ of pairwise edge-disjoint $(\eps, s)$-expanders that covers all but at most $48sn (\log n)^{3}$ edges of $G$ and satisfies $\sum_{H \in \cH}|H| \le 2n$. Let $G_1$ be the subgraph of $G$ spanned by edges not covered by $\cH$. Then $e(G_1) \le n(\log n)^{55}$ and the graphs in $\cH$ are weak $(\frac{\eps}{2}, s)$-expanders.

		Apply \Cref{cor:sep-dense-expander} (with $\frac{\eps}{2}$ and $s$) to each $H \in \cH$ to obtain a collection $\cP_H$ of at most $636 |H|$ paths that separates $H$. Let $\cQ_H$ be a collection of at most $|H|$ paths that decomposes $H$ (which exists by \Cref{thm:lovasz}). Take $\cP := \bigcup_{H \in \cH} (\cP_H \cup \cQ_H)$. 

		We claim that $\cP$ separates $G \setminus G_1$. 
		To see this, consider two edges $e, f$ in $G \setminus G_1$ and let $H \in \cH$ be a graph containing $e$. If $f \in H$ then $\cP_H$ has a path containing $e$ but not $f$, and otherwise $\cQ_H$ has such a path.

		To finish, note the following bound.
		\begin{align*}
			|\cP| & \le \sum_{H \in \cH} (|\cP_H| + |\cQ_H|) 
			\le \sum_{H \in \cH} 637|H| 
			\le 1300 n.
		\end{align*}
		So the graph $G_1$ and the collection of paths $\cP$ satisfy the requirements of the lemma.
	\end{proof}

\section{Expanding with forbidden sets} \label{sec:limited-contact}

	As a preparation for the proof of the second key lemma, in this section we prove a variant of Lemma 3.2 from \cite{liu2020solution}, which allows one to expand a set of vertices $A$ while avoiding another set $X$, provided that $X$ has a small intersection with each ball around $A$ (in \cite{liu2020solution} the authors list two other properties of $X$ that guarantee the expansion of $A$ while avoiding $X$, but we shall not use them). In both cases, the conclusion is that the radius-$i$ ball around $A$ in $G \setminus X$ is large. In \cite{liu2020solution} this was only shown for specific, fairly large values of $i$, whereas here we give a general lower bound that can be applied for any $i$, as per our needs later on.

	Following \cite{liu2020solution}, say that a vertex set $A$ has \emph{$k$-limited contact} with a vertex set $X$ in a graph $G$ if $\left|N_G(B_{H - X}^{i-1}(A)) \cap X\right| \le ki$ for every integer $i \ge 1$.
	The following lemma is similar to Lemma 3.2 in \cite{liu2020solution}, but tailored to our needs.
	\begin{lemma} \label{lem:expand-limited-contact}
		For every $\eps > 0$ there exists $k_0 = k_0(\eps)$ such that the following holds for $k \ge k_0(\eps)$.
		Suppose that $G$ is an $n$-vertex $\eps$-expander, and let $A$ and $X$ be disjoint subsets of $V(G)$ where $|A| \ge k^3$ and $A$ has $k$-limited contact with $X$. Then $\left|B_{G - X}^{i}(A)\right| > \min\{2^{i^{1/4}}, n/2\}$, for every integer $i \ge 1$.
	\end{lemma}

	\begin{proof}
		Write $H :=  G - X$ and $B_i := B_{H}^i(A)$ for $i \ge 0$. Our task is thus to show $|B_{i}| > \min\{2^{i^{1/4}}, n/2\}$ for $i \ge 1$. 
		\begin{claim} \label{claim:Bi}
			For every integer $i \ge 0$, either $|B_i| > n/2$ or $|N_{H}(B_i)| \ge \frac{\eps|B_i|}{2(\log|B_i|)^2}$.
		\end{claim}
		\begin{proof}
			We prove the claim by induction on $i$. Let $i \ge 0$ be such that $|B_i| \le n/2$ and suppose that the statement holds for $j$ with $0 \le j < i$. (We will prove the induction base and step simultaneously.)

			For $j \in [i]$, since $B_j$ is the disjoint union of $B_{j-1}$ and $N_H(B_{j-1})$, we have 
			\begin{align*}
				|B_j| \ge \left(1 + \frac{\eps}{2(\log|B_{j-1}|)^2}\right)|B_{j-1}|  
				\ge \left(1 + \frac{\eps}{2(\log|B_i|)^2}\right)|B_{j-1}|, 
			\end{align*}
			using that $B_{j-1} \subseteq B_i$.
			Iterating this, we get
			\begin{equation} \label{eqn:Bi}
				|B_i| \ge \left(1 + \frac{\eps}{2(\log|B_i|)^2}\right)^i |A|.
			\end{equation}
			(Notice that this inequality holds for $i = 0$, as $B_0 = A$, since $A$ is assumed to be disjoint of $X$, allowing us to treat the induction base and step simultaneously.)

			We conclude that $|B_i| \ge \max\{k^3, i^3\}$. Indeed, if $i \le k$ this holds because $|B_i| \ge |A| \ge k^3$. So suppose that $i > k$ and $|B_i| < i^3$. Then \eqref{eqn:Bi} implies 
			\begin{equation*}
				|B_i| \ge \left(1 + \frac{\eps}{2(\log i^3)^2}\right)^i|A|
				\ge \left(\exp\left(\frac{\eps}{4(\log i^3)^2}\right)\right)^i
				= \exp\left(\frac{\eps i}{36(\log i)^2}\right)
				\ge i^3,
			\end{equation*}
			using that $i$ is large and thus $\frac{\eps}{2(\log i^3)^2} \le 1$ (which follows from $k$ being large and the assumption $i > k$). This is a contradiction, proving that $|B_i| \ge \max\{k^3, i^3\}$.
			Write $m := \max\{k, i\}$. Then 
			\begin{equation} \label{eqn:Bi-vs-ki}
				\frac{\eps|B_i|}{(\log|B_i|)^2} \ge \frac{\eps m^3}{(\log m)^2} \ge 4m^2 \ge 2k(i+1),
			\end{equation}
			using that $m$ is large.

			Recall that $|X \cap N_G(B_i)| \le k(i+1)$ (this is what it means for $A$ to have $k$-limited contact with $X$) and notice that $N_G(B_i) - X \subseteq N_H(B_i)$. Thus, 
			\begin{align*}
				|N_H(B_i)| \ge |N_G(B_i)| - k(i+1) 
				\ge \frac{\eps |B_i|}{(\log |B_i|)^2} - k(i+1)
				\ge \frac{\eps|B_i|}{2(\log|B_i|)^2},
			\end{align*}
			where the second inequality follows from expansion (using $|B_i| \le n/2$), and the third inequality follows from \eqref{eqn:Bi-vs-ki}.
			This proves the induction hypothesis for $i$.
		\end{proof}

		Suppose that $i \ge 1$ satisfies $|B_i| \le \min\{2^{i^{1/4}}, n/2\}$. Since $|B_i| \ge |A| \ge k^3$, it follows that $2^{i^{1/4}} \ge k^3$. In particular, $i$ is large. By \Cref{claim:Bi}, every $j \in [i]$ satisfies
		\begin{align*}
			|B_j| = |B_{j-1}| + |N_H(B_{j-1})|
			\ge \left(1 + \frac{\eps}{2(\log|B_{j-1}|)^2}\right)|B_{j-1}|
			\ge \left(1 + \frac{\eps}{2\sqrt{i}}\right)|B_{j-1}|,
		\end{align*}
		using that $B_{j-1} \subseteq B_i$ and the assumption $|B_i| \le \min\{2^{i^{1/4}}, n/2\}$.
		Iterating this, we get
		\begin{align*}
			|B_i| \ge \left(1 + \frac{\eps}{2\sqrt{i}}\right)^i
			\ge \exp\left( \frac{\eps i}{4\sqrt{i}}\right)
			> 2^{i^{1/4}},
		\end{align*}
		using that $i$ is large, and yielding a contradiction.
	\end{proof}

\section{Proof of \Cref{lem:reduce-small-deg}: separating sparse expanders} \label{sec:sparse-expanders}

	In this section we prove the second key lemma, \Cref{lem:reduce-small-deg}. The proof will be split into two parts: \Cref{lem:bound-deg}, proved in \Cref{subsec:large-degree}, and \Cref{lem:sparse-expanders}, proved in \Cref{subsec:small-max-deg}. In both parts, we are presented with an expander $G$ on $n$ vertices and a number $d$, and we wish to separate edges of $G$ by $O(n + e(G)/d)$ paths. In the first part, we separate edges that touch large degree vertices, and in the second part we consider expanders with small maximum degree and many vertices.
	\Cref{lem:reduce-small-deg} itself is proved in \Cref{subsec:proof-lemma-small-deg}.

	\subsection{Dealing with large degree vertices} \label{subsec:large-degree}

		The following lemma shows that given an expander $G$ on $n$ vertices and a number $d$, there is a collection of $O(n + e(G)/d)$ paths that separate all edges of $G$ except for a remainder with maximum degree at most $d^7$. Take $L_1$ to be the set of vertices of degree at least $d^7$ and $L_2$ to be the set of vertices outside of $L_1$ that send at least four edges into $L_1$. We will let the remainder graph be $G - L_1$. Noting that we can separate the edges of $G[L_1, V(G) \setminus (L_1 \cup L_2)]$ by $O(n)$ single edges, it suffices to separate the edges in $G' := G[L_1] \cup G[L_1, L_2]$. 

		To do so, as usual we first decompose $G'$ into $O(n + e(G)/d)$ paths $\cP$ of length at most $d$. Then, for each $u \in L_2$ we pick a family $\cA_u$ of $d(u, L_1)$ paths of length $2$ with the central vertex being $u$ and the other two vertices in $L_1$, such that each edge between $u$ and $L_1$ is covered twice, and each such $2$-path intersects each path in $\cP$ in at most one edge. Write $\cA = \left(\bigcup_{u \in L_2}\cA_u\right) \cup E(G[L_1])$, so that $\cA$ is a collection of paths of length at most $2$. Then, we find a collection $\cM$ of $O(n + e(G)/d)$ graphs $M$, which are disjoint unions of at most $d$ paths of length at most $2$ from $\cA$, such that $\cP \cup \cM$ separates $G'$.
		Now it suffices to show that each $M \in \cM$ can be extended to a path that avoids using other edges in $E := \bigcup_{e \in E(M)} P(e)$ (where $P(e)$ is the unique path in $\cP$ that contains $e$). Take $F$ to be a linear forest that extends $M$, avoids other edges in $E$, and minimises the number of components and then edges under these assumptions. We show that $F$ must be a path, using \Cref{lem:expand-limited-contact} and that all leaves in $F$ have degree at least $d^7$ in $G$, and are thus easy to expand.

		\begin{lemma} \label{lem:bound-deg}
			Let $\eps > 0$ and let $d$ be sufficiently large. 
			Suppose that $G$ is an $n$-vertex $\eps$-expander. Then there is a subgraph $G_1 \subseteq G$ with maximum degree at most $d^7$ and a collection $\P$ of at most $17\max\{n, e(G)/d\}$ paths in $G$ that separates the edges of $G \setminus G_1$.
		\end{lemma}

		\begin{proof}
			Write $m := e(G)$ and $k := \max(m/d, n)$.

			Let $L_1$ be the set of vertices in $G$ with degree at least $d^7$ and let $L_2$ be the set of vertices outside of $L_1$ that have at least four neighbours in $L_1$. Define $G_1 = G \setminus L_1$. We will find a collection $\P$ of at most $17\max\{n, e(G)/d\}$ paths in $G$ that separate the edges of $G \setminus G_1$.

			Write $H_1 := G[L_1] \cup G[L_1, L_2]$ and $H_2 := G[L_1, V(G) - (L_1 \cup L_2)]$. Let $\P_1$ be a collection of at most $n + m/d \le 2k$ paths of length at most $d$ that decomposes $H_1$; such a collection exists by \Cref{cor:lovasz}. 
			For an edge $e$ in $H_1$, let $P(e)$ be the unique path in $\cP_1$ that contains $e$.

			For each $u \in L_2$, let $v_1, \ldots, v_{d_u}$, where $d_u := d(u, L_1)$, be an ordering of the neighbours of $u$ in $L_1$ such that $uv_i$ and $uv_{i+1}$ are not in the same path in $\P_1$, for $i \in [d_u]$ (indices taken modulo $d_u$). Notice that such an ordering exists because $d_u \ge 4$ and the paths in $\P_1$ are pairwise edge-disjoint. Define $\cA_u := \{v_i u v_{i+1} : i \in [d_u]\}$ (with indices taken modulo $d_u$) for $u \in L_2$ and $\cA := \left(\bigcup_{u \in L_2} \cA_u\right) \cup E(G[L_1])$, so $\cA$ is a collection of edges and $2$-paths.

			\begin{claim}
				There is a collection $\cM$ of at most $12k$ graphs which are disjoint unions of at most $d$ paths of lengths at most $2$ from $\cA$, such that each path in $\cA$ is a component in exactly one graph $M \in \cM$; and every path in $\P_1$ shares at most one edge with any $M \in \cM$.
			\end{claim}
			\begin{proof}
				Let $\cM$ be a collection of $12k$ graphs that are disjoint unions of at most $d$ paths from $\cA$, such that: each $A \in \cA$ is a component in at most one $M \in \cM$; every path in $\P_1$ shares at most one edge with any $M \in \cM$; and a maximum number of paths in $\cA$ are components in some graph in $\cM$.
				We claim that every edge and $2$-path in $\cA$ is a component in some $M \in \cM$. Indeed, suppose not, and let $A \in \cA$ be an edge or $2$-path which is not a component in any $M \in \cM$.

				Notice that $|\cA| \le m$, so there are at most $m/d$ graphs $M \in \cM$ with $d$ components.

				We claim that every vertex $v \in L_1 \cup L_2$ is in at most $2n$ paths in $\cA$. Indeed, if $v \in L_1$ then there are at most $|L_1|$ edges in $\cA$ that contain $v$ and at most $2|L_2|$ paths of length $2$ in $\cA$ that contain $v$, because each edge in $G[L_1, L_2]$ is used by exactly two such $2$-paths. 
				If $v \in L_2$ then $v$ is in at most $2|L_1|$ paths of length $2$ in $\cA$ (and in no edges in $\cA$). Either way, there are at most $2|L_1| + 2|L_2| \le 2n$ paths in $\cA$ that contain $v$.
				As $|V(A)| \le 3$, it follows that $A$ shares a vertex with at most $6n$ paths in $\cA$. In particular, there are most $6n$ graphs $M \in \cM$ such that $A$ and $M$ share a vertex.

				Notice also that each path in $\P_1$ has at most $n$ edges and thus intersects the edges of at most $2n$ paths in $\cA$. Thus, there are at most $4n$ graphs $M \in \cM$ for which there is a path $P \in \P_1$ that shares an edge with both $A$ and $M$.

				In summary, since $m/d + 6n + 4n \le 11k < 12k$, there is $M \in \cM$ such that $M$ has fewer than $d$ components, $A$ and $M$ do not share a vertex, and there is no path $P \in \P_1$ that shares an edge with $A$ and $M$. Thus we can replace $M$ by $M \cup A$, contradicting the maximality of $\cM$.
			\end{proof}

			Let $\cM$ be a collection as guaranteed by the above claim. 
			For a graph $M \in \cM$, write $E_{M} = \bigcup_{e \in E(M)} E(P(e))$ and $X_M := V(E_M) \setminus V(M)$ (so $E_M$ is the set of edges that appear on paths from $\cP_1$ that share edges with $M$, and $X_M$ is the set of vertices on such paths, minus the vertices of $M$). Note that $|X_M| \le 2d^2$ (there are at most $2d$ edges in $M$, each of which causes the insertion of at most $d$ vertices into $X_M$). 

			\begin{claim} \label{claim:complete-paths}
				There is a path in $G$ that contains $M$ and avoids the edges in $E_M \setminus E(M)$, for $M \in \cM$.
			\end{claim}

			\begin{proof}
				Fix $M \in \cM$. For convenience, write $E := E_M$ and $X := X_M$.
				Let $\Q$ be a collection of pairwise vertex-disjoint paths of length at least $2$ in $G - X$ with the following properties.
				\begin{enumerate}[label = \rm(\alph*)]
					\item \label{itm:Q-struct}
						The union of $M$ and the paths in $\Q$ is a path forest whose components are paths that alternate between components of $M$ and paths in $\Q$, starting and ending with a path in $M$. In particular, $|\Q| \le d$.
					\item \label{itm:Q-max}
						Subject to \ref{itm:Q-struct}, $|\Q|$ is maximal.
					\item \label{itm:Q-min}
						Subject to \ref{itm:Q-struct} and \ref{itm:Q-max}, $\sum_{Q \in \Q}\ell(Q)$ is minimal.
				\end{enumerate}
				Denote the path forest, which is the union of $M$ and paths in $\Q$, by $F$. 

				Let $R$ be a component in $F$, let $u$ be an end of $R$, and let $X_u := (X \cup V(F)) - \{u\}$. We claim that the following holds for $r \ge 1$. 
				\begin{equation} \label{eqn:u-limited-contact}
					\left|N_{G}\!\left(B_{G - X_u}^{r-1}(u)\right) \cap V(F) \right| \le (r + 4)d.
				\end{equation}
				Suppose this is not the case. Because $|V(M)| \le 3d$ and $|\Q| \le d$, there is a path $Q \in \Q$ such that
				\begin{equation} \label{eqn:u-Q}
					\left|N_{G}\!\left(B_{G - X_u}^{r-1}(u)\right) \cap V(Q^{\circ}) \right| > \frac{(r+4)d - 3d}{d} \ge r+1,
				\end{equation}
				where $Q^{\circ}$ is the interior of $Q$.
				Let $R'$ be the component of $F$ that contains $Q$. Let $v$ be one of the ends of $Q$; if $R' = R$ take $v$ to be the further end of $Q$ from $u$ in $R$ (see \Cref{fig:two-paths,fig:one-path}).
				Let $w$ be the first vertex in $Q$ (starting from $v$) in $N_{G}(B_{G \setminus X_u}^{r-1}(u)) \cap V(Q^{\circ})$, and let $Q_1$ be the subpath of $Q$ that starts at $v$ and ends at $w$. 
				Take $Q_2$ to be a shortest path in $G$ from $u$ to $w$ whose interior vertices are not in $V(F) \cup X$. By choice of $w$ we have $\ell(Q_2) \le r$. Finally, let $Q'$ be the concatenation of $Q_1$ and $Q_2$ and define $\Q' := (\Q - \{Q\}) \cup \{Q'\}$. 
				\begin{figure}[ht]
					\centering
					\includegraphics[scale = 1]{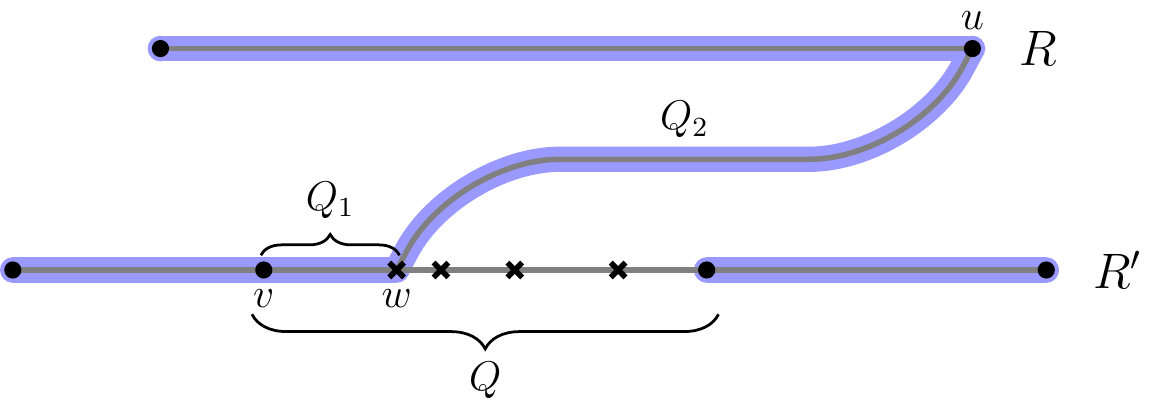}
					\vspace{-.4cm}
					\caption{The case $R' \neq R$. The crosses signify the vertices in $N_{G}\!\left(B_{G - X_u}^{r-1}(u)\right) \cap V(Q^{\circ})$, and the blue paths replace $R$ and $R'$ in the path forest $M \cup \cQ'$.}
					\label{fig:two-paths}
				\end{figure}
				\begin{figure}[ht]
					\centering
					\includegraphics[scale = 1]{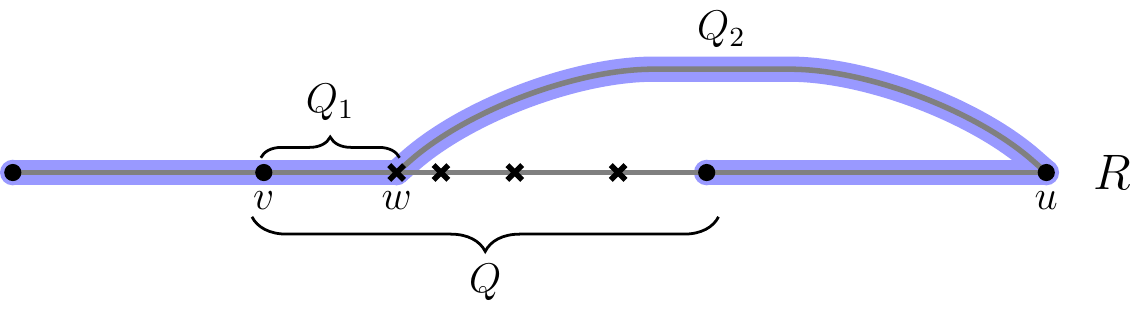}
					\vspace{-.2cm}
					\caption{The case $R' = R$. The crosses are as above, and the blue path replaces $R$ in the path forest $M \cup \cQ'$.}
					\label{fig:one-path}
				\end{figure}
				
				Notice that the paths in $\Q'$ are pairwise vertex-disjoint paths of length at least $2$ in $G - X$ (in particular, $Q'$ was indeed chosen to have length at least $2$ and be in $G - X$). 
				Notice also that $\Q'$ satisfies \ref{itm:Q-struct}. Indeed, if $R' = R$ then by choice of $v$ we have that $(R - Q) \cup Q'$ is a path, and if $R' \neq R$ then $((R \cup R') - Q) \cup Q'$ is a vertex-disjoint union of two paths. Other components of $F$ are not affected by the switch from $\Q$ to $\Q'$.
				Clearly, $|\Q'| = |\Q|$. Finally, notice that $\ell(Q') = \ell(Q_1) + \ell(Q_2) \le \ell(Q_1) + r < \ell(Q_1) + \ell(Q - Q_1)$, where the last inequality holds because $Q - Q_1$ is a path with at least $r$ interior vertices, due to \eqref{eqn:u-Q}. This implies that $\sum_{Q \in \Q'} \ell(Q) < \sum_{Q \in \Q} \ell(Q)$, contradicting \ref{itm:Q-min}, and thus proving \eqref{eqn:u-limited-contact}.

				We now conclude that $F$ is a path, namely it has exactly one component. Suppose not, and let $R_1$ and $R_2$ be two components in $F$. Let $u_j$ be an end of $R_j$, for $j \in [2]$. By \ref{itm:Q-struct}, both $u_1$ and $u_2$ are ends of paths in $M$, and thus have degree at least $d^7$ in $G$. Let $A_j := N_G(u_j) - (X \cup V(F))$.
				Then 
				\begin{equation*}
					|A_j| \ge |N_G(u_j)| - |X| - |N_{G}(u_j) \cap V(F)|
					\ge d^7 - 2d^2 - 5d \ge (9d^2)^3,
				\end{equation*}
				using \eqref{eqn:u-limited-contact} (with $r = 1$), $|X| \le 2d^2$ and that $d$ is large. Also,
				\begin{equation*}
					\left| N_G\Big(B_{G - V(F) - X}^{r-1}(A_j)\Big) \cap \big(V(F) \cup X\big) \right| \le (r+5)d + 2d^2 + 1 \le 9rd^2, % brackets
				\end{equation*}
				using \eqref{eqn:u-limited-contact} (with $r+1$), where the $2d^2 + 1$ term accounts for $X \cup \{u\}$. % explain more explicitly?
				In particular, $A_j$ has $9d^2$-limited contact with $V(F)$ in $G'$ (the notion of limited contact was defined at the beginning of \Cref{sec:limited-contact}). Take $r = (\log n)^4$. Then, by \Cref{lem:expand-limited-contact}, $|B^r_{G - V(F) - X}(A_j)| > n/2$ for $j \in [2]$.  Hence, there is a path $Q$ in $G - V(F) - X$ joining $A_1$ and $A_2$. Let $Q'$ be the extension of $Q$ by edges from $u_j$ to $A_j$, for $j \in [2]$. 
				Take $\Q' := \Q \cup \{Q'\}$. Then $\Q'$ is a collection of pairwise vertex-disjoint paths of length at least $2$ in $G - X$ that satisfy \ref{itm:Q-struct} and $|\Q'| > |\Q|$, contradicting \ref{itm:Q-max}. This proves that $F$ consists of a single component. 

				We claim that $F$ satisfies the requirements of the claim. Clearly, $F$ contains $M$. Moreover, any edge in $F$ which is not in $E(M)$ is incident with at least one vertex which is not in $X_M \cup V(M) = V(E_M)$ and thus it is not in $E_M$.
			\end{proof}
			For $M \in \cM$, denote by $P_M$ the paths whose existence is guaranteed by \Cref{claim:complete-paths}.
			To prove the theorem, take $\P$ to be the union of $\P_1$, $\P_2$ and the edges of $H_2 = G[L_1, V(G) - (L_1 \cup L_2)]$ (each edge in $H_2$ is considered as a separate path in $\P$).
			Notice that $|\P| = |\P_1| + |\P_2| + e(H_2) \le 2k + 12k + 3n \le 17k = 17\max\{n, e(G)/d\}$. 
			
			We claim that $\P$ separates all edges in $G$ that touch $L_1$. To see this, consider two distinct edges, $e$ and $f$, that touch $L_1$. 
			If $e \in E(H_2)$ then $e$ itself is a path in $\cP$ (which obviously does not contain $f$). Otherwise, either the path $P(e)$ contains $e$ but not $f$, or there exists $M \in \cM$ that contains $e$ but not $f$ and then $P_M$ contains $e$ but not $f$.

			Recall that $G_1 = G - L_1$. Then $G_1$ has maximum degree at most $d^7$, and the previous paragraph shows that $\P$ separates the edges of $G \setminus G_1$. This proves \Cref{lem:bound-deg}.
		\end{proof}

	\subsection{Sparse expanders with small maximum degree} \label{subsec:small-max-deg}

		The next lemma shows that given a `somewhat robust' expander $G$ on $n$ vertices with maximum degree at most $d^7$, where $n$ is large in terms of $d$, the edges of $G$ can be separated by $O(n + e(G)/d)$ paths.
		The proof is very similar in structure to the proof of \Cref{lem:bound-deg}. We first decompose $G$ into paths $\cP$ and find matchings $\cM$, such that $|\cP|, |\cM| = O(n + e(G)/d)$, the union $\cP \cup \cM$ separates $G$, the paths in $\cP$ and matchings in $\cM$ have size at most $d$, and for any two edges $e, f$ in the same matching, the paths $P(e)$ and $P(f)$ are far apart in $G$ (we need $n$ to be large with respect to $d$ to be able to satisfy the last property).

		Now it remains to show that for each $M \in \cM$ there is a path that contains $M$ and avoids $E := \bigcup_{e \in M}P(e) \setminus E(M)$.
		As before, we take $F$ to be a path forest that extends $M$, avoids edges in $E$, and minimises the number of components and then edges. We again use \Cref{lem:expand-limited-contact} to deduce that $F$ is a path. This time we do so by first showing that small-radius balls around any leaf of $F$ expand well, using that such balls have few edges of $E$ (by assumption on $M$) and the moderate robustness property. We can then expand such balls further similarly to the proof of \Cref{lem:bound-deg}, using that we now start from sizeable sets.

		\begin{lemma} \label{lem:sparse-expanders}
			Let $\eps > 0$, let $d$ be sufficiently large, let $n \ge 2^{(\log d)^7}$, and let $s \ge (k_0(\eps)_{\ref{lem:expand-limited-contact}} + 3)^3$.
			Suppose that $G$ is an $n$-vertex $(\eps, s, d)$-expander with maximum degree at most $d^7$. Then there is a collection $\P$ of at most $5\max\{n, e(G)/d\}$ paths in $G$ that separates the edges of $G$.
		\end{lemma}

		\begin{proof}
			\def \rr {r_0}
			As before, write $m := e(G)$ and $k := \max\{n, m/d\}$.
			Let $\P_1$ be a collection of at most $n + m/d \le 2k$ paths of length at most $d$ that decompose the edges in $G$; it exists by \Cref{cor:lovasz}.
			For an edge $e$ denote by $P(e)$ the unique path in $\P_1$ that contains $e$.
			Write $\rr = (\log d)^5$.

			\begin{claim}
				There is a decomposition $\cM$ of $G$ of size $3k$ such that every $M \in \cM$ satisfies: $M$ is a matching of size at most $d$; and the paths $P(e)$ and $P(f)$ are a distance at least $2\rr$ apart for every distinct $e, f \in M$.
			\end{claim}

			\begin{proof}
				Let $\cM$ be a collection of $3k$ pairwise edge-disjoint matchings that satisfy the conditions of the claim, and together cover a maximum number of edges in $G$. We claim that $\cM$ covers all edges of $G$. Suppose not, and let $e \in E(G) - \bigcup \cM$.

				As $e(G) = m$, there are at most $m/d$ matchings $M \in \cM$ of size $d$.

				We have the following upper bound on the ball of radius $2\rr$ around $P(e)$.
				\begin{equation*}
					\big|B_G^{2\rr}(P(e))\big| \le |P(e)| \cdot (1 + d^7 + \ldots + (d^7)^{2\rr}) \le d^{14\rr + 2},
				\end{equation*}
				using the maximum degree assumption, that $|P(e)| \le d+1$, and that $d$ is large. 
				Notice that every $P \in \P_1$ intersects the edges of at most $d$ matchings $M \in \cM$, implying that each edge $f$ is in $\bigcup_{h \in M} P(h)$ for at most $d$ matchings $M \in \cM$. It follows from the maximum degree assumption that every vertex $u$ is in $\bigcup_{h \in M}P(h)$ for at most $d^8$ matchings $M \in \cM$.
				Hence, the sets $B_G^{2\rr}(P(e))$ and $\bigcup_{h \in M}P(h)$ share a vertex, for at most $d^8 \cdot d^{14\rr+2} \le d^{14\rr+10} \le 2^{(\log d)^7} \le n$ matchings $M$, where for the last inequality we used the assumption on $n$ and $d$. In other words, $P(e)$ is at distance at least $2\rr$ from $\bigcup_{h \in M}P(h)$, for all but at most $n$ matchings $M \in \cM$.

				To summarise, since $m/d + n < 3k$, there is a matching $M \in \cM$ such that $|M| < d$ and $P(e)$ and $\bigcup_{h \in M}P(h)$ are a distance at least $2\rr$ apart. We can thus replace $M$ by $M \cup \{e\}$, a contradiction to the maximality of $\cM$.
			\end{proof}

			Let $\cM$ be a decomposition of $G$ into matchings, as guaranteed by the above claim. For an edge $e$, denote by $M(e)$ the unique matching in $\cM$ that contains $e$.

			\begin{claim} \label{claim:complete-b}
				There is a path in $G$ that contains all edges in $M$ and avoids the edges in $\bigcup_{e \in M} (P(e) - \{e\})$, for $M \in \cM$.
			\end{claim}

			\begin{proof}
				Fix $M \in \cM$, and write $E := \bigcup_{e \in M} P(e)$ and $G' := G - E$.
				Let $\Q$ be a collection of pairwise vertex-disjoint paths in $G'$ with the following properties.
				\begin{enumerate} [label = \rm(\alph*)]
					\item \label{itm:Qb-struct}
						The union of $M$ with the paths in $\Q$ is a path forest whose paths alternate between $M$ and $\Q$, starting and ending with an edge from $M$.
					\item \label{itm:Qb-max}
						Subject to \ref{itm:Qb-struct}, $|\Q|$ is maximal.
					\item \label{itm:Qb-min}
						Subject to \ref{itm:Qb-struct} and \ref{itm:Qb-max}, $\sum_{Q \in \Q}\ell(Q)$ is minimal.
				\end{enumerate}
				Denote the path forest which is the union of $M$ and $\Q$ by $F$.

				Let $u$ be an end of a path in $F$, let $v$ be the neighbour of $u$ in $F$ (so $uv \in M$), let $R$ be the component in $F$ that contains $u$, and write $X_u := V(F) - \{u\}$. We claim the following, for $r < \rr$,
				\begin{equation} \label{eqn:limited-contact-short}
					\left| N_G\left(B_{G' - X_u}^{r-1}(u)\right) \cap V(F) \right| \le r+1.
				\end{equation}
				Suppose this is violated for some $r \le \rr$. As edges in $M$ are a distance at least $2\rr$ apart in $G$, the set $N_G(B_{G' - X_u}^{r-1}(u)) \cap V(M)$ is contained in $\{v\}$, so it has size at most $1$. Hence there is a path $Q \in \Q$ such that either $v$ is not an end of $Q$ and $N_G(B_{G' - X_u}^{r-1}(u)) \cap V(Q^{\circ}) \neq \emptyset$ or $v$ is an end of $Q$ and $\big|N_G(B_{G' - X_u}^{r-1}(u)) \cap V(Q^{\circ})\big| \ge r+1$.

				Suppose that the former holds, and let $Q \in \Q$ be a path, none of whose ends is $v$, that satisfies $N_G(B_{G' - X_u}^{r-1}(u)) \cap V(Q^{\circ}) \neq \emptyset$. Let $w_0$ be a vertex in this intersection, let $Q_0$ be a shortest path in $G'$ from $u$ to $w_0$ whose interior avoids $X_u$; so $\ell(Q_0) \le r < \rr$.
				Denote the ends of $Q$ by $w_1$ and $w_2$, where if $Q$ is in $R$ then $w_1$ is further away from $u$ in $R$. Let $Q_i$ be the subpath of $Q$ from $w_0$ to $w_i$, for $i \in [2]$. Notice that $\ell(Q_i) \ge \rr$, because of the assumption that edges in $M$ are a distance at least $2\rr$ apart in $G$; indeed, otherwise the concatenation $Q_0Q_i$ of $Q_0$ and $Q_i$ is a path of length shorter than $2\rr$ between two distinct edge in $M$.
				Write $\Q' := \Q - \{Q\} \cup \{Q_0Q_1\}$.
				It is easy to check that $\Q'$ is a collection of vertex-disjoint paths in $G'$ that satisfies \ref{itm:Qb-struct}, $|\Q'| = |\Q|$, and $\sum_{Q \in \Q'} \ell(Q) < \sum_{Q \in \Q} \ell(Q)$, contradicting \ref{itm:Qb-min}. 

				It remains to consider the second case, namely that $Q \in \Q$ has ends $w_1$ and $w_2$, where $w_2 = v$, and $\big|N_G(B^{r-1}_{G' - X_u}(u)) \cap V(\Q^{\circ})\big| \ge r+1$. Let $w_0$ be the vertex in the intersection which is furthest away from $w_2$ in $Q$ and let $Q_0$ be a shortest path from $u$ to $w_0$ in $G'$ whose interior avoids $X_u$; so $\ell(Q_0) \le r$.
				Let $Q_i$ be the subpath of $Q$ from $w_0$ to $w_i$, for $i \in [2]$. Then $\ell(Q_2) \ge r + 1$ by assumption on the intersection size $\big|N_G(B^{r-1}_{G' - X_u}(u) \cap V(\Q^{\circ})\big|$. 
				Defining $\Q' := \Q - \{Q\} \cup \{Q_0w_0Q_1\}$, we reach a contradiction as in the previous paragraph. Thus \eqref{eqn:limited-contact-short} is proved.
				
				We now conclude the following. 
				\begin{equation} \label{eqn:expand-u-small}
					\left| B_{G' - X_u}^{\rr}(u) \right| \ge 2^{(\log d)^2} \ge d^8.
				\end{equation}
				Write $G'' := G - P(uv)$. We claim that $B_{G' - X_u}^r(u) = B_{G'' - X_u}^r(u)$ for every $r \le \rr$. Indeed, this holds because $G'' = G' - \bigcup_{f \in M - \{uv\}} P(f)$, and $P(f)$ is a distance at least $2\rr$ away from $uv$ in $G$ for every $f \in M - \{uv\}$. In particular, \eqref{eqn:limited-contact-short} implies
				\begin{equation} \label{eqn:G''}
					\left| N_{G''}\left(B^{r-1}_{G'' - X_u}(u)\right) \cap X_u \right| \le r + 1,
				\end{equation}
				for $r < \rr$.
				We claim further that $G''$ is an $\eps$-expander. To see this, let $X \subseteq V(G'')$ satisfy $1 \le |X| \le 2n/3$. Let $E''$ be the set of edges from $P(uv)$ that touch $X$. Then $|E''| \le \min\{2|X|, d\}$ and $N_{G''}(X) = N_{G - E''}(X)$. Thus, as $G$ is an $(\eps, s, d)$-expander and $s \ge 2$, $|N_{G''}(X)| = |N_{G - E''}(X)| \ge \frac{\eps|X|}{(\log|X|+1)^2}$, as required for $\eps$-expansion. Note also that $G$ has minimum degree at least $s+1$, by $(\eps, s, d)$-expansion, implying that $G''$ has minimum degree at least $s-1$. Write $A_u := N_{G'' - X_u}(u) \cup \{u\}$. Then \eqref{eqn:G''} implies $|A_u| \ge |N_{G''}(u)| - 2 \ge s-3$ and $\big|N_{G''}(B^{r-1}_{G'' - X_u}(A_u)) \cap X_u\big| \le r+2 \le 3r$ for $r \le \rr-2$. Write $X_u' := X_u \cap N_{G''}(B^{\rr-3}_{G'' - X_u}(u))$. Then $A_u$ has $3$-limited contact with $X_u'$. Using that $s$ is large, \Cref{lem:expand-limited-contact} shows $|B_{G'' - X_u'}^{\rr-2}(A_u)| \ge 2^{(\rr-2)^{1/4}} \ge 2^{(\log d)^2} \ge d^8$. This proves \eqref{eqn:expand-u-small}, as $B_{G''-X_u}^{\rr-1}(u) = B_{G'' - X_u'}^{\rr-2}(A_u)$.
				
				Next, we follow the proof of \Cref{claim:complete-paths} quite closely to show that the set $B_{G' - X_u}^{\rr}(u)$ expands.
				Define $V := \bigcup_{e \in M} V(P(e))$, $B_u := B_{G' - X_u}^{\rr}(u) - V$ and $Y := V \cup V(F)$. Notice that $|B_u| \ge d^8 - d(d+1) \ge d^7$ by \eqref{eqn:expand-u-small} and since $|V| \le d(d+1)$. We claim the following for $r \ge 1$.
				\begin{equation} \label{eqn:limited-contact-long}
					\left| N_G\!\left(B^{r-1}_{G - Y}(B_u)\right) \cap V(F) \right| \le d(r + \rr + 1).
				\end{equation}
				The proof of this is very similar to that of \eqref{eqn:u-limited-contact} in \Cref{claim:complete-paths}.
				Suppose this does not hold for some $r \ge 1$. Because $|V(M)| \le 2d$ and $|\Q| \le d$, there is a path $Q \in \Q$ such that
				\begin{equation*}
					\big| N_G(B^{r-1}_{G - Y}(B_u)) \cap V(Q^{\circ}) \big| \ge \frac{d(r + \rr + 3) - 2d}{d} \ge r+\rr+1.
				\end{equation*}
				Let $R'$ be the component in $F$ that contains $Q$. Denote the ends of $Q$ by $w_1$ and $w_2$, where $w_1$ is further from $u$ in $R$ if $R' = R$. Let $w_0$ be the vertex in $N_G(B^{r-1}_{G - Y}(B_u)) \cap V(Q^{\circ})$ which is closest to $w_1$ in $Q$, and let $Q_0$ be a shortest path in $G'$ from $u$ to $w_0$ whose interior avoids $Y$; so $\ell(Q_0) \le r+\rr$. 
				Let $Q_i$ be the subpath of $Q$ from $w_0$ to $w_i$, for $i \in [2]$; then $\ell(Q_2) \ge r+\rr+1$. Taking $\Q' := \Q - \{Q\} \cup \{Q_0 w_0 Q_1\}$ yields the desired contradiction to \ref{itm:Qb-min}. Thus \eqref{eqn:limited-contact-long} is proved.

				Next, we deduce the following.
				\begin{equation} \label{eqn:expand-u-large}
					\big| B_{G - Y}^{(\log n)^4}(B_u) \big| > n/2.
				\end{equation}
				Notice that \eqref{eqn:limited-contact-long} implies $\big|N_G(B_{G - Y}^{r-1}(B_u)) \cap Y\big| \le d(r + \rr + d + 2) \le 3d^2r$ for $r \ge \rr$ (using that $d$ is large), i.e.\ $B_u$ has $3d^2$-limited contact with $Y$ in $G$. As $|B_u| \ge d^7 \ge (3d^2)^3$, \Cref{lem:expand-limited-contact} proves \eqref{eqn:expand-u-large}.

				Finally, we conclude that $F$ consists of a single component; namely, it is a path. Suppose not, and let $u_1$ and $u_2$ be ends of distinct components in $F$. By \eqref{eqn:expand-u-large}, the sets $B^{(\log n)^4}_{G - Y}(B_{u_1})$ and $B^{(\log n)^4}_{G - Y}(B_{u_2})$ intersect, showing that there is a path $Q$ in $G'$ from $u_1$ to $u_2$ whose interior avoids $V(F)$. 
				Taking $\Q' := \Q \cup \{Q'\}$, we reach a contradiction to \ref{itm:Qb-max}. 
				We have thus proved that $F$ is a path. As it contains $M$ and has no other edges in $\bigcup_{e \in M} P(e)$, it satisfies the requirements of \Cref{claim:complete-b}. 
		\end{proof}
		For $M \in \cM$, let $P_M$ be a path as guaranteed by \Cref{claim:complete-b}. Write $\cP_2 := \{P_M : M \in \cM\}$, and take $\P := \P_1 \cup \P_2$. Then $|\P| \le |\P_1| + |\P_2| \le 5k$. Notice that $\P$ separates the edges of $G$. Indeed, given $e, f \in E(G)$, then one of $P(e)$ and $M(e)$ contains $e$ but not $f$.
		Thus \Cref{lem:sparse-expanders} is proved.
	\end{proof}

	\subsection{Proof of \Cref{lem:reduce-small-deg}} \label{subsec:proof-lemma-small-deg}

		Finally, we put everything together to prove \Cref{lem:reduce-small-deg}, restated here. The proof follows quite straightforwardly from \Cref{lem:bound-deg,lem:sparse-expanders}.

		\lemReduceSmallDeg*

		\begin{proof}[Proof of \Cref{lem:reduce-small-deg}]
			Let $G$ be a graph on $n$ vertices with average degree $d$, where $d$ is large. Take $\eps := 1/48$ and apply \Cref{lem:decompose-expanders} with $\eps$, $s = 0$ and $t = 2n/3$ (the last choice does not matter as $t$ does not play a role when $s = 0$) to obtain a collection $\HH$ of $\eps$-expanders that decomposes $G$ and satisfies $\sum_{H \in \HH} |H| \le 2n$.
			
			Apply \Cref{lem:bound-deg} to each $H \in \HH$, to obtain a subgraph $F_H \subseteq H$ and a collection $\P_H$ of paths in $H$ such that: $F_H$ has maximum degree at most $d^7$; $|\P_H| \le 17 \max\{|H|, e(H)/d\}$; and $\P_H$ separates the edges of $H \setminus F_H$. Let $\P'_H$ be collection of at most $|H|$ paths decomposing $H \setminus F_H$ (which exists by \Cref{thm:lovasz}).

			Now, for each $H \in \HH$, apply \Cref{lem:decompose-expanders} to $F_H$ with $\eps$, $s_{\ref{lem:decompose-expanders}} = \left((k_0(\eps))_{\ref{lem:expand-limited-contact}} + 3\right)^3$ and $t_{\ref{lem:decompose-expanders}} = d$, to obtain a collection $\F_H$ of pairwise edge-disjoint subgraphs of $F_H$ that cover all but at most $48s|F_H|(\log d + 1)^2$ edges of $F_H$ and satisfy $\sum_{F \in \F_H}|F| \le 2|F_H|$.
			Denote by $E_H$ the set of edges in $H$ uncovered by $\F_H$. 

			Let $\F_1$ be the collection of graphs $F \in \bigcup_{H \in \HH}\F_H$ that satisfy $|F| \ge 2^{(\log d)^7}$ and let $\F_2 := (\bigcup_{H \in \HH}F_H) - \F_1$. Apply \Cref{lem:sparse-expanders} to each $F \in \F_1$ to obtain a collection $\Q_F$ of paths in $F$ such that $|\Q_F| \le 5 \max\{|F|, e(F) / d\}$ and $\Q_F$ separates the edges of $F$.

			Take $\P$ to be the union of the collections $\P_H, \P'_H$, for $H \in \HH$, and $\Q_F$, for $F \in \F_1$. Define $G_1$ to be the graph on vertices $V(G)$ and edges $\bigcup_{H \in \HH} E_H$. We show that the statement of \Cref{lem:reduce-small-deg} holds (with $\HH_{\ref{lem:reduce-small-deg}} = \F_2$). Notice that the graphs in $\F_2$ indeed are pairwise edge-disjoint, because the graphs in $\HH$ are pairwise edge-disjoint and so are the graphs in $\F_H$ for every $H \in \HH$. Moreover, $|F| \le 2^{(\log d)^7}$ for every $F \in \F_2$, by choice of $\F_2$.
			Finally,
			\begin{align*}
				\sum_{F \in \F_1 \cup \F_2}|F| 
				= \sum_{H \in \HH} \sum_{F \in \F_H} |F|
				\le \sum_{H \in \HH} 2|F_H|
				\le \sum_{H \in \HH} 2|H|
				\le 4n,
			\end{align*}
			using the choice of $\HH$ and $\F_H$. In particular, $\sum_{F \in \F_2} |F| \le 4n$. Thus \ref{itm:key-two-HH} is proved.

			Next, notice that
			\begin{align*}
				|\P| 
				& \le \sum_{H \in \HH} \left(|\P_H| + |\P'_H|\right) + \sum_{F \in \F_1} |Q_F| \\
				& \le \sum_{H \in \HH} 18(|H| + e(H)/d) + \sum_{F \in \F_1}5(|F| + e(F)/d) \\
				& \le 36n + 18e(G)/d + 20n + 5e(G)/d
				\le 80n.
			\end{align*}
			For the second inequality we used that the graphs in $\HH$ are pairwise edge-disjoint and similarly for $\F_1$, and for the last inequality we used that $G$ has average degree $d$.

			We claim that $\P$ separates the edges of $G' := G - G_1 - \bigcup_{F \in \F_2} F$.
			Indeed, let $e$ and $f$ be distinct edges in $G'$, and let $H \in \HH$ satisfy $e \in H$.  
			If $e \in H \setminus F_H$, then if $f \notin H \setminus F_H$ there is a path in $\cP_H'$ that contains $e$ but not $f$, and otherwise there is such a path in $\cP_H$. Now suppose $e \in F_H$ and let $F \in \F_H$ be a graph that contains $e$. Then $\cQ_F$ has a path that contains $e$ but not $f$. This proves \ref{itm:key-two-P}.

			Finally, we have
			\begin{align*}
				2e(G_1) 
				\le 2\sum_{H \in \HH} |E_H|
				\le \sum_{H \in \HH} 96s|F_H|(\log d + 1)^2
				\le \frac{1}{2}\sum_{H \in \HH} |H| (\log d)^3
				\le n (\log d)^3,
			\end{align*}
			using that $d$ is large and $s$ is a constant (it depends only on $\eps$, which itself is a constant). This completes the proof of \ref{itm:key-two-G1} and \Cref{lem:reduce-small-deg}.
		\end{proof}

\section{Conclusion} \label{sec:conc}

	Recall that $\wsep(G)$ is the size of a smallest weakly-separating path system for $G$, and $\wsep(n)$ is the maximum of $\wsep(G)$ over all $n$-vertex graphs $G$; $\sep(G)$ and $\sep(n)$ are defined analogously for strong separation.

	In this paper we proved that $\sep(n) = O(n \logstar n)$. This is significant progress from the easy initial bound of $O(n \log n)$, towards a conjecture of, independently, Falgas-Ravry--Kittipassorn--Kor\'andi--Narayanan and the author \cite{falgas2013separating}, and Balogh--Csaba--Martin--Pluh\'ar \cite{balogh2016path}, that $\sep(n) = O(n)$.

	It is plausible that our methods could be used to fully settle this conjecture. For example, one could try to prove a statement like the following, by induction on $n$. Suppose that $G$ is an $\eps$-expander and let $H$ be a subgraph of $G$ on $n$ vertices. Then the edges of $H$ can be separated using $O(n)$ paths in $G$. To accomplish this, it would suffice to find $O(n)$ paths in $G$ that separated all edges of $H$ touching a set $V$ of $\Omega(n)$ vertices in $H$. An obstruction to achieving this is a graph $H$ as follows: $H$ has average degree $d$, but almost all vertices in $H$ have degree much lower than $d$ (say, at most $\log \log d$), and, moreover, there are relatively small sets of low degree vertices in $H$ with few low degree neighbours.
	
	In fact, the authors of \cite{falgas2013separating} suggest that, perhaps, $\wsep(n) = (1 + o(1))n$; it is easy to see that this bound would be tight by considering complete graphs. If true, this is likely to be very hard to prove. Indeed, even the restriction to complete graph seems hard. Wickes \cite{wickes2022separating} proved that $\wsep(K_n) \le (\frac{21}{16} + o(1))n$ for any integer $n$, and that $\wsep(K_n) \le n$ when $n = p \text{ or } p+1$ for some prime $p$. It would be nice, and perhaps not completely out of reach, to prove that $\wsep(K_n) = (1 + o(1))n$. It also makes sense to study $\sep(K_n)$.

\subsection*{Acknowledgements}

	I would like to thank Ant\'onio Gir\~ao for reminding me of this problem and making the connection to sublinear expansion.
	I am grateful to Matija Buci\'c and Richard Montgomery for insightful discussions regarding their very nice paper \cite{bucic2022erdos} and for giving me a copy of an early draft.
	Some of the research leading up to this paper was done while the author attended the MFO workshop Combinatorics, Probability and Computing, held in Oberwolfach in April 2022, and the BIRS workshop Cross Community Collaborations in Combinatorics, held in Banff in June 2022. I am thankful to the organisers of both workshops for inviting me, and to MFO and BIRS for providing great infrastructure for research.
	I would also like to thank the kind anonymous referee for carefully reading a previous version of this paper, and for their various helpful suggestions.

\bibliography{sep}
\bibliographystyle{amsplain}

\end{document}